\documentclass{lms}

\usepackage{amsmath}
\usepackage{amssymb}
\usepackage{color}
\usepackage{verbatim}
\usepackage[line,color,poly,all]{xy}
\usepackage{graphicx}

\usepackage[pdftex, pdfusetitle, pdfpagelabels, plainpages=false,
  bookmarks, bookmarksnumbered
]{hyperref} 

\numberwithin{equation}{section}

\newtheorem{theorem}[equation]{Theorem} 
\newtheorem{lemma}[equation]{Lemma}     

\newnumbered{assertion}[equation]{Assertion}    
\newnumbered{conjecture}[equation]{Conjecture}  
\newnumbered{definition}[equation]{Definition}
\newnumbered{hypothesis}[equation]{Hypothesis}
\newnumbered{remark}[equation]{Remark}
\newnumbered{note}[equation]{Note}
\newnumbered{observation}[equation]{Observation}
\newnumbered{problem}[equation]{Problem}
\newnumbered{question}[equation]{Question}
\newnumbered{algorithm}[equation]{Algorithm}
\newnumbered{example}[equation]{Example}
\newunnumbered{notation}{Notation} 

\DeclareMathOperator{\Cl}{Cl}

\DeclareMathOperator{\Gal}{Gal}
\DeclareMathOperator{\GL}{GL}

\DeclareMathOperator{\impart}{Im}

\DeclareMathOperator{\M}{M}
\DeclareMathOperator{\N}{N}
\DeclareMathOperator{\nrd}{nrd}

\DeclareMathOperator{\repart}{Re}

\DeclareMathOperator{\trd}{trd}
\DeclareMathOperator{\Tr}{Tr}
\DeclareMathOperator{\vol}{vol}

\newcommand{\C}{\mathbb C}

\newcommand{\PP}{\mathbb P}

\newcommand{\Q}{\mathbb Q}
\newcommand{\R}{\mathbb R}
\newcommand{\Z}{\mathbb Z}

\newcommand{\calD}{\mathcal D}
\newcommand{\calO}{\mathcal O}
\newcommand{\calN}{\mathcal N}

\newcommand{\eps}{\epsilon}

\newcommand{\frakb}{\mathfrak{b}}

\newcommand{\frakl}{\mathfrak{l}}

\newcommand{\frakM}{\mathfrak{M}}
\newcommand{\frakp}{\mathfrak{p}}
\newcommand{\frakq}{\mathfrak{q}}
\newcommand{\frakn}{\mathfrak{n}}
\newcommand{\frakd}{\mathfrak{d}}
\newcommand{\frakN}{\mathfrak{N}}

\newcommand{\calH}{\mathcal{H}}

\newcommand     {\abs}[1]       {{\left\lvert{#1}\right\rvert}}

\newcommand{\legen}[2]{\left(\frac{#1}{#2}\right)}

\newcommand{\quat}[2]{\displaystyle{\biggl(\frac{#1}{#2}\biggr)}}

\newcommand{\la}{\langle}
\newcommand{\ra}{\rangle}

\newcommand{\defi}[1]{\textbf{\textsf{#1}}} 				

\title[Twists of Hilbert modular forms]
 {Nonvanishing of twists of $L$-functions \\ attached to Hilbert
  modular forms}

\author{Nathan C.\ Ryan, Gonzalo Tornar\'ia, and John Voight}

\classno{11F41, 11M50 (primary), 11F67 (secondary)}

\extraline{The first author and second authors were supported by
  grant CSIC I+D 2012-628
  from the Comisi\'on Sectorial de Investigaci\'on Cient\'ifica
  The third author was supported by an NSF CAREER Award (DMS-1151047).}

\begin{document}
\maketitle

\begin{abstract}
We describe algorithms for computing central values of twists of
$L$-functions associated to Hilbert modular forms, carry out 
such computations for a number of examples, and compare the results of these computations 
to some heuristics and predictions from random matrix theory.  
\end{abstract}



\section{Introduction}

Let $f(q)=\sum_{n=1}^{\infty} a_n q^n \in S_k(N)$ be a classical newform of weight $k$ and level $N$, and let $\lambda_n=a_n/\sqrt{n}^{k-1}$.  For $D<0$ a fundamental discriminant, let 
\[ L(f,s,\chi_D)=\sum_{n=1}^{\infty} 
\chi_D(n)\frac{\lambda_n}{n^s} \]
be the $L$-series of $f$ (in the analytic normalization) twisted by
the quadratic character $\chi_D$ associated to the imaginary quadratic
field $\Q(\sqrt{D})$.  The central values $L(f,1/2,\chi_D)$ encode
interesting arithmetic information about the form $f$, and a number of explicit
investigations have been carried out examining the family of these values \cite{Gross,BSP,PT2,PT1,MRVT}.  

An efficient way to compute the family $L(f,1/2,\chi_D)$ of central values with varying discriminant $D$ is to use Waldspurger's theorem \cite{Waldspurger}, which asserts that the values are related to the Fourier
coefficients of a certain half-integer weight modular form.
Concretely, for $D<0$ coprime to $N$, we have
\begin{equation} \label{eqn:LfchiD}
L(f,1/2,\chi_D) = \kappa_f  \frac{c_{|D|}(g)^2}{\sqrt{\abs{D}}^{k-1}}
\end{equation}
where the (nonzero) constant $\kappa_f$ is independent of $D$ and $c_{|D|}(g)$
is the $|D|$th coefficient of a
modular form $g$ of weight $(k+1)/2$ related to $f$ via the Shimura correspondence.
Computing central values using \eqref{eqn:LfchiD} has the advantage that the description of $g$ as a linear combination of theta series permits the rapid computation of a large number of coefficients: for example, Hart--Tornar\'ia--Watkins \cite{hart2010congruent} compute hundreds of billions of twists of the congruent number elliptic curve (using FFT methods).  By comparison, experiments with the distribution of twists with similar Hodge data computed \emph{without} using Waldspurger's theorem are much less extensive: see e.g.\ Watkins \cite[\S 6.6]{MR2410120} and David--Fearnley--Kisilevsky \cite{MR2322350}.

Several authors have pursued Waldspurger's theorem in the setting of Hilbert modular forms, including
Shimura \cite{Shimura}, Baruch--Mao \cite{BaruchMao}, Xue \cite{Xue}, Sirolli \cite{Sirolli}, and Hiraga--Ikeda \cite{HiragaIkeda}.  In this paper, we develop algorithms using these
formulas to compute families of central values $L(f,1/2,\chi_D)$ for 
Hilbert modular forms $f$ over totally real fields $F$, defined in an analogous way.  
We believe that these computations are of independent interest,
but we also use them to provide some partial evidence for conjectures concerning statistics for central
values in families of twists motivated by some heuristics and refined by random matrix theory.  

Let $F$ be a totally real field of degree $n=[F:\Q]$ with ring of integers $\Z_F$, and suppose that $F$ has narrow class number $1$.  Let
\[ \calD(\Z_F) = \{D \in \Z_F/\Z_F^{\times 2} : \text{$D$ fundamental discriminant and } D \ll 0\} \]
be the set of totally negative fundamental discriminants in $\Z_F$; the set $\calD(\Z_F)$ is in canonical bijection with the CM extensions $F(\sqrt{D})$ of $F$.  For $X>0$, let
\[ \calD(\Z_F;X) = \{D \in \calD(\Z_F) : |N_{F/\Q}(D)| \leq X\}. \]
To $D \in \calD(\Z_F)$, let $\chi_D$ be the character associated to the quadratic extension $F(\sqrt{D})$. 

Let $f$ be a Hilbert cusp form over $F$ of parallel weight $k \in 2\Z_{>0}$ and
level $\frakN \subseteq \Z_F$ with rational integer Hecke eigenvalues, and let
$w_f$ be the sign of the functional equation for $L(f,s)$.  
We will be interested in the number of vanishings
\begin{equation} \label{eqn:NfF}
\calN_f(\Z_F;X)=\#\bigl\{D \in \calD(\Z_F;X) \;:\;
    \text{$\chi_D(\frakN)=(-1)^n\,w_f$ and $L(f,1/2,\chi_D)=0$}\bigr\} 
\end{equation}
as a function of $X$; the condition that $\chi_D(\frakN)=(-1)^n\,w_f$ is equivalent to condition that the sign of the functional
equation for $L(f,s,\chi_D)$ is $+1$, so that $L(f,s,\chi_D)$ vanishes to even
order.  

\begin{conjecture}\label{conj:ZF}
There exist $b_f,C_f\geq 0$ depending on $f$ such that as $X \to \infty$, we have
\[ \calN_f(\Z_F;X) \sim C_f\/X^{1-(k-1)/4}\,(\log X)^{b_f}. \]
\end{conjecture}

When $k=2$, the modular form $f$ is expected (and known in many cases)
to correspond to an isogeny class of elliptic curves $E$ over $F$, so according
to the conjecture of Birch--Swinnerton-Dyer, Conjecture \ref{conj:ZF} predicts
the distribution of curves of even rank $\geq 2$ in CM quadratic twists of $E$
over $F$.  
Conjecture \ref{conj:ZF} generalizes the conjecture of
Conrey--Rubinstein--Keating--Snaith \cite{CKRS} made in the case $F=\Q$,
thinking of $L(f,s)$ as a degree 2 $L$-function over $F$.  The additional power $b_f$ of $\log X$ that appears is conjecturally
related to arithmetic and geometric properties of $f$: for
example, if $F=\Q$ and $k=2$, then Delaunay--Watkins \cite{MR2322344}
conjecture a value for $b_f$ that depends on the two-torsion structure of
elliptic curves in the isogeny class associated to $f$ (see
Section~\ref{sec:rmt} below for more discussion).  The constant $C_f \geq 0$ is
not presently understood; when $k \geq 6$, we predict that $C_f=0$.

Our second conjecture is a variant of the first and investigates a
new phenomenon that arises in the context of Hilbert modular form, restricting to twists by discriminants $D \in \Z$; we still write $\chi_D$ for the quadratic character of $F(\sqrt{D})$ over $F$.   Let
\begin{equation} \label{eqn:NfZF}
 \calN_{f}(\Z;X)=\#\bigl\{D \in \calD(\Z;X) \;:\;
    \text{$\chi_D(\frakN)=(-1)^n\,w_f$ and $L(f,1/2,\chi_D)=0$}\bigr\}. 
\end{equation}

\begin{conjecture}\label{conj:Z}
There exist $b_{f,\Z},C_{f,\Z}\geq 0$ depending on $f$ such that as $X \to \infty$, we have
\[ \calN_{f}(\Z;X) \sim C_{f,\Z}\, X^{1-n(k-1)/4}\,(\log X)^{b_{f,\Z}}. \]
\end{conjecture}

In Conjecture \ref{conj:Z}, we instead are thinking of $L(f,s)$ as a degree $2n$ $L$-function over $\Q$, and as a result, vanishing twists (with $D \in \Z$) are more rare.  Put another way, the discriminants $D \in \Z$ are sparse among all discriminants $D \in \Z_F$, and consequently their contribution to the number of vanishing twists is scant.  The power $b_{f,\Z}$ of $\log X$ depends on $f$ as in Conjecture 1.2 but with an additional factor coming from the expected number of primes in the factorization in $\Z_F$ of $(p)$ for a prime $p\in\Z$.

\begin{remark}
In the case of non-parallel weight $(k_1,\dots,k_n)$, we expect
Conjectures~\ref{conj:ZF} and \ref{conj:Z} hold
replacing $k$ with the average of the $k_i$.
\end{remark}

This paper is organized as follows.  We begin in Section~\ref{sec:background} by giving some background on Hilbert modular forms and their $L$-functions, and we state a
version of Waldspurger's theorem for Hilbert modular forms.  In Section~\ref{sec:algs}, we make this theory algorithmic, and 
exhibit methods to compute a large number of central values of $L$-functions of quadratic twists over totally real fields.  In Section~\ref{sec:rmt}, we describe the heuristics motivating Conjectures \ref{conj:ZF}--\ref{conj:Z} and show how they might be refined using the connections between random matrix theory and $L$-functions in our context.  In Section~\ref{sec:comps}, we describe the experiments we carried out and present some data, tables and graphs.  Our computations are done primarily in \textsf{Magma} \cite{MAGMA}.   We conclude in Section \ref{sec:ques} with some remaining questions.

\section{Background and notation}\label{sec:background}

In this section, we summarize the background and introduce the notation we use throughout.  General references for Hilbert modular forms include Freitag \cite{Freitag}, van der Geer \cite{geer}, and Goren \cite{Goren}; for algorithmic aspects, see the survey of Demb\'el\'e--Voight \cite{DembeleVoight}.  

\subsection*{Hilbert modular forms}

Let $F$ be a totally real number field of degree $n=[F:\Q]$,
discriminant $d_F$, ring of integers $\Z_F$, and different $\frakd$.
We assume throughout that $F$ has narrow class number 1; removing this hypothesis is possible but would make the exposition more technical.  Let $v_1,\dots,v_n$ be the real embeddings
of $F$ into $\R$. Let $\frakN$ be a
nonzero ideal of $\Z_F$.  Let 
\[ F_+=\{a \in F : v_i(\alpha)>0 \text{ for all $i=1,\dots,n$}\} \]
and for a fractional ideal $\frakb$ of $F$, let $\frakb_+=\frakb \cap F_+$.  Let $\calH=\{z \in \C : \impart(z)>0\}$ denote the upper half-plane.

A \defi{Hilbert modular form} over $F$ of \defi{weight} $k=(k_i)_i \in 2\Z_{\geq 0}^n$ and \defi{level} $\frakN$ is a
holomorphic function $f:\calH^n\to\C$ such that for all
$z=(z_1,\dots,z_n)\in \calH^n$ and all
\[
\gamma\in\Gamma_0(\frakN)=\left\{\gamma=\begin{pmatrix}a&b\\c&d\end{pmatrix}\in
\GL_2(\Z_F):c\in\frakN\text{ and }\det(\gamma) \in \Z_{F,+} \right\},
\]
we have
%
\begin{equation} \label{eqn:fgammazfull}
f(\gamma z)=f\left(\frac{a_1z_1+b_1}{c_1z_1+d_1},\dots,\frac{a_n z_n+b_n}{c_nz_n+b_n}\right)=\left(\prod_{i=1}^n\frac{(c_iz_i+d_i)^{k_i}}{\det\gamma_i^{k_i/2}}\right) f(z),
\end{equation}
where $\gamma_i=v_i(\gamma)=\begin{pmatrix}a_i&b_i\\c_i&d_i\end{pmatrix}$, and such that further $f$ satisfies a condition of bounded growth (necessary only in the case $F=\Q$, by Koecher's principle).  If $k_i=k$ for all $i=1,\dots,n$, we say $f$ has \defi{parallel weight} $k \in 2\Z_{\geq 0}$.  A \defi{Hilbert cusp form} is a Hilbert modular form that vanishes at the \defi{cusps}, elements of $\PP^1(F) \hookrightarrow \PP^1(\R)^n$.  We denote by $S_k(\frakN) \subseteq M_k(\frakN)$ the space of Hilbert cusp forms inside the finite-dimensional space of Hilbert modular forms of weight $k$ and level $\frakN$ over $F$.

A Hilbert modular form $f \in M_k(\frakN)$ admits a Fourier expansion
\begin{equation} \label{eqn:fourier}
f(z) = a_0 +\sum_{\mu\in (\frakd^{-1})_+}a_\mu q^{\Tr(\mu z)}
\end{equation}
where $q^{\Tr(\mu z)}=\exp(2\pi i\sum_{i=1}^n \mu_i z_i)$, analogous to the $q$-expansion of a classical modular form.  If $f(z) \in S_k(\frakN)$, then $a_0=0$.

The space $S_k(\frakN)$ is equipped with the Petersson inner product given by
\begin{equation} \label{eqn:Petersson}
\la f,g\ra =\frac{1}{\vol(X_0(\frakN))}\int_{X_0(\frakN)}f(z)\overline{g(z)}y^{k-2}\, dx\, dy
\end{equation}
where $X_0(\frakN)=\Gamma_0(\frakN)\backslash \calH^n$, $z_i = x_i + \sqrt{-1} y_i$ and $dx=dx_1 \cdots dx_n$, $dy=dy_1\cdots dy_n$.  Further, the spaces $M_k(\frakN)$ and $S_k(\frakN)$ are equipped with an
action of pairwise commuting Hecke operators $T_\frakn$, indexed by the nonzero ideals $\frakn$ of $\Z_F$ with $\frakn$ coprime to $\frakN$, that are self-adjoint with respect to the Petersson inner product. 
These spaces are also equipped with Atkin-Lehner operators $W_{\frakp^e}$ for prime powers $\frakp^e \parallel \frakN$.
If $f\in S_k(\frakN)$ is a (simultaneous) eigenform, we say $f$ is \defi{normalized} if $a_{(1)}=1$; then $T_\frakn
f=a_\frakn f$ and each eigenvalue $a_\frakn$ is an algebraic integer
that lies in a totally real number field $K_f=\Q(\{a_\frakn\})$.  

There are natural injections $S_k(\frakM) \hookrightarrow S_k(\frakN)$ when
$\frakM \mid \frakN$, and we define a \defi{newform} to be a normalized cusp
eigenform which is orthogonal to the image of any such injection with $\frakM$ a
proper divisor of $\frakN$.  For each unitary divisor $\frakq \parallel \frakN$
(so $\frakq \mid \frakN$ and $\frakq$ is coprime to $\frakN/\frakq$), there is
an Atkin--Lehner operator $W_{\frakq}$ that is an involution on the new
subspace of $S_k(\frakN)$, so that $W_{\frakq} f = w_{\frakq,f} f = \pm f$ for
a Hecke newform $f$.

\subsection*{Theta series}

The theory of half-integral weight modular forms was developed by Shimura in a series of papers \cite{shimura1985eisenstein,shimura1987hilbert,shimura1993theta}.  Suppose that $\frakN$ is \defi{odd} (i.e., $\N(\frakN)$ is odd).  A \defi{Hilbert modular form} over $F$ of \defi{half-integer weight} $k \in (1/2,\dots,1/2)+\Z_{\geq 0}^n$, \defi{level} $4\frakN$, and quadratic \defi{character} $\psi$ of conductor dividing $4\frakN$ is a
holomorphic function $f:\calH^n\to\C$ such that for all
$z \in \calH^n$ and all $\gamma=\begin{pmatrix} a & b \\ c & d \end{pmatrix} \in \Gamma_0(4\frakN)$, we have
\[
f(\gamma z) = \psi(d)h(\gamma,z)\left(\prod_{i=1}^n\frac{(c_iz_i+d_i)^{k_i}}{\det\gamma_i^{k_i/2}}\right) f(z)
\]
as in \eqref{eqn:fgammazfull}; the factor $h(\gamma,z)$ is an automorphy factor of weight $(1/2,\dots,1/2)$ defined on the metaplectic cover of $\Gamma_0(4\frakN)$ by 
\[ h(\gamma,z)=\frac{\theta(\gamma z)}{\theta(z)}, \quad \text{ where } \theta(z)=\sum_{a \in \Z_F} q^{\Tr(a^2 z)}. \]
In this paper, the half-integral weight modular forms we encounter will arise from linear combinations of weighted theta series, as follows.

Let $Q:V \to F$ be a positive definite quadratic form on a $F$-vector space $V$ with $\dim_F V=d$.  Let $\Lambda \subset V$ be a \defi{(full) $\Z_F$-lattice}, so $\Lambda$ is a finitely generated $\Z_F$-submodule that contains a basis for $V$.  Suppose that $\Lambda$ is \defi{integral}, so $Q(\Lambda) \subseteq \Z_F$.  Let $\frakN$ be the discriminant of $\Lambda$ and suppose that $\frakN$ is odd.

A homogeneous polynomial $P(x)=P(x_1,\dots,x_d)$ on $\R^d$ is \defi{spherical harmonic}
if
\[ \left(\frac{\partial^2}{\partial x_1^2} + \dots + \frac{\partial^2}{\partial x_d^2}\right) P = 0. \]
For $i=1,\dots,n$, let $P_i(x)$ be a spherical harmonic polynomial on $V\otimes_{v_i} \R$
of homogeneous degree $m_i$, and let $P(x)=\prod_{i=1}^n P_i(x)$.  We define the \defi{theta series} associated to $\Lambda$ and $P$ as
\begin{equation} \label{eqn:ThetaLambdaP} 
\begin{aligned}
\Theta(\Lambda,P) : \calH^n &\to \C \\
\Theta(\Lambda,P;z) &= \sum_{x \in \Lambda} P(x)\,q^{\Tr(Q(x)z)}.
\end{aligned}
\end{equation}
Then $\theta$ is a Hilbert modular form of level $4\frakN$, weight $(d/2+m_i)_i$, and quadratic character $\psi$ of conductor dividing $4\frakN$ \cite{MR1246183}.
Note that $\sum_i m_i$ must be even, or else trivially we have $\Theta(\Lambda,P)=0$.

\subsection*{Twists of $L$-functions}  

For the analytic properties of Hilbert modular forms we use in this section, we refer to Shimura \cite{MR507462}.
Let $f \in S_k(\frakN)$ be a Hilbert newform of weight $k$ and level $\frakN$,
and let $k_0=\max(k_1,\dots,k_n)$.
Let $T_\frakn f = a_\frakn f$ and
let $\lambda_\frakn = a_\frakn/\sqrt{\N(\frakn)}^{k_0-1}$.

Associated to $f$ is an $L$-function $L(f,s)$ given by the normalized Dirichlet
series
\[
L(f,s)=\sum_{\frakn}\frac{\lambda_\frakn}{\N(\frakn)^s},
\]
convergent for $\repart s>1$.  Define
the completed $L$-function
\begin{equation} \label{eqn:lambdacomp} 
\Lambda(f,s) = Q^{s/2} \left(\prod_{i=1}^n \Gamma_\C(s+(k_i-1)/2))\right) \cdot L(f,s)
\end{equation}
where $Q=d_F^2\N(\frakN)$ is the \defi{conductor} and
$\Gamma_\C(s)=2(2\pi)^{-s}\Gamma(s)$.  Then $\Lambda(f,s)$ has an analytic continuation to $\C$ and satisfies the functional equation
\[
\qquad\qquad
\Lambda(f,s)=w_f\Lambda(f,1-s)
\qquad\qquad
\text{with $w_f=(-1)^{(k_1+\dotsb+k_n)/2}\, w_{\frakN,f}$,}
\]
where $w_{\frakN,f}$ is the canonical (Atkin-Lehner) eigenvalue of $f$.

We will be interested in twisting $L$-functions by quadratic characters, which are indexed by fundamental discriminants, as follows.  The ring of integers of a quadratic field $K$ is of the form $\Z_K=\Z_F[x]/(x^2-tx+n)$ with $t,n \in \Z_F$; the quantity $D=t^2-4n \in \Z_F / \Z_F^{\times 2}$ is the \defi{fundamental discriminant} uniquely associated to $K=F(\sqrt{D})$.  Consequently, there is a bijection between the set of isomorphism classes of quadratic fields $K/F$ and fundamental discriminants $D$.  An element $D \in \Z_F$ is a fundamental discriminant if and only if $D$ is a square modulo $4$ and minimal with respect to this property under divisibility by squares: i.e., we have $D \equiv t^2 \pmod{4}$ for some $t \in \Z_F$ and,
if $D=f^2 d$ with $d,f \in \Z_F$ and $d$ a square modulo $4$, then $f \in \Z_F^{\times}$.  Let
\[ \calD(\Z_F) = \{D \in \Z_F/\Z_F^{\times 2} : \text{$D$ fundamental discriminant and } D \ll 0\} \]
be the set of totally negative fundamental discriminants.  For $D \in \calD(\Z_F)$, let
$\chi_D$ be the quadratic character associated to $K=F(\sqrt{D})$.  

Suppose that $D$ is coprime to $\frakN$.  We consider the quadratic twist $f_D = f \otimes \chi_D \in S_k(\frakN D^2)$ of $f$ and its associated $L$-function $L(f,s,\chi_D)$ defined by
\[
L(f,s,\chi_D)=L(f_D,s)=\sum_{\frakn}\chi_D(\frakn)\frac{\lambda_\frakn}{\N(\frakn)^{s}}.
\]
The completed $L$-function $\Lambda(f,s,\chi_D) = \Lambda(f_D,s)$ is defined as
in \eqref{eqn:lambdacomp} with $Q=d_F^2\N(\frakN D^2)$;
it satisfies the functional equation
\[
\qquad\qquad
\Lambda(f,s,\chi_D) = w_{f_D}\,\Lambda(f,1-s,\chi_D)
\qquad\qquad
\text{with $w_{f_D} = (-1)^n\,\chi_D(\frakN)\,w_f$.}
\]

\begin{definition} \label{defn:parity}
A weight satisfies the \defi{parity condition} if $(k_1+\dotsb+k_n)/2\equiv n\pmod{2}$.
We say that $D \in \calD(\Z_F)$ is \defi{permitted} if $D$ is coprime to $\frakN$ and 
$\chi_D(\frakq)=w_{\frakq,f}$ for all $\frakq \parallel \frakN$.
\end{definition}

The parity condition ensures that
\begin{equation}\label{eqn:sign}
  w_f=(-1)^n\,w_{\frakN,f},
\end{equation}
and $w_{f_D} = \chi_D(\frakN)\,w_{\frakN,f}$.
For permitted $D$
we have $\chi_D(\frakN)=w_{\frakN,f}$, thus
$w_{f_D}=1$.  Note that the parity condition for parallel weight $k\in 2\Z_{>0}$
holds when $n$ is even or when $n$ is odd and $k\equiv 2\pmod{4}$.  



To simplify notation, we write $|D|=-D$ when $D$ is totally negative.

\begin{conjecture} \label{conj:Waldspurger}
Let $f \in S_k(\frakN)$ be a Hilbert newform of odd squarefree level $\frakN$ such that $k$ satisfies the parity condition.
Then there exists a modular form $g(z)=\sum_{\mu\in (\frakd^{-1})_+}c_\mu q^{\Tr(\mu z)} \in S_{(k+1)/2}(4\frakN)$ such that for all permitted $D \in \calD(\Z_F)$, we have
\begin{equation}\label{eqn:Waldspurger}
  L(f,1/2,\chi_D) = \kappa_f\frac{c_{|D|}(g)^2}{\prod_{i=1}^n \sqrt{|v_i(D)|}^{k_i-1}},
\end{equation}
where $\kappa_f \neq 0$ is independent of $D$.
In the case of parallel weight $k$ the denominator
in the right hand side
is just $\sqrt{\N(D)}^{k-1}$.
\end{conjecture}

Conjecture \ref{conj:Waldspurger} is known in many cases.  It follows in principle from general work of Waldspurger \cite{Waldspurger} and Shimura \cite{Shimura}, but we were not able to extract the explicit statement above: the difficult part is to construct a nonzero form $g$.  Baruch--Mao \cite{BaruchMao} prove Conjecture \ref{conj:Waldspurger} in the case $F=\Q$ (using work of Kohnen for the existence of $g$).  Xue \cite{Xue} gives a proof for general $F$ and prime level $\frakN=\frakp$, provided $L(f,1/2) \neq 0$; Hiraga--Ikeda \cite{HiragaIkeda} treat general $F$ and trivial level $\frakN=(1)$.  We will discuss in Section~\ref{sec:algs} an algorithm that conjecturally always computes the form $g$ associated to $f \in S_k(\frakN)$ when $L(f,1/2)\neq 0$: it is correct as long as $g \neq 0$, by the work of the above authors.

\begin{remark}
The condition that $\frakN$ is odd and squarefree are important here for the purposes of a concise exposition; already in the case $F=\Q$, to relax these hypotheses involves technical complications.  Similarly, there are extensions to totally real fundamental discriminants as well, but we do not pursue them here.
The restriction to permitted discriminants is necessary; if $\legen{D}{\frakl} = - w_{\frakl,f}$ for some prime $\frakl \parallel \frakN$
then the coefficient $c_{|D|}$ is trivially zero but $L(f,1/2,\chi_D)$ need not be zero.
\end{remark}

\section{Algorithms}\label{sec:algs}

In this section, we discuss algorithms to test the vanishing or nonvanishing of central values $L(f,1/2,\chi_D)$ using the theory of the previous section, encoded as the coefficients of a linear combination of quaternionic theta series.  We continue with our notation from the previous section: in particular, let $f \in S_k(\frakN)$ be a Hilbert newform over $F$.  
Throughout, we employ conventions and algorithms described by Kirschmer--Voight \cite{MR2592031,MR2967473}; a general reference for the results on quaternion algebras we will use is Vign\'eras \cite{MR580949}.

As we will be employing an algorithmic version of Conjecture \ref{conj:Waldspurger}, we assume that $\frakN$ is odd and squarefree and that the weight $k$ satisfies the parity condition (Definition \ref{defn:parity}).  Moreover, we make the assumption that $L(f,1/2) \neq 0$ to work with the simplest construction of the associated half-integral weight form $g$: for extensions, see Remark \ref{rmk:constL0} below.  

\subsection*{Brandt module}

The assumption that $L(f, 1/2)\neq 0$ implies, in particular, that $w_f=1$.
By \eqref{eqn:sign}, we have $w_{\frakN,f}=(-1)^n$;
in other words
$\#\{\frakl\mid\frakN \;:\; w_{\frakl,f}=-1\}\equiv n\pmod{2}$,
and it follows that there is a
quaternion algebra $B$ over $F$ ramified at all the real places
and at primes $\frakl\mid\frakN$ with $w_{\frakl,f}=-1$.
Since $\frakN$ is squarefree there is an Eichler order
$\calO \subset B$ with reduced discriminant $\frakN$.

The quaternion algebra $B$ has a unique involution $\overline{\phantom{x}}:B \to B$ such that the \defi{reduced norm} $\nrd(\alpha)=\alpha\overline{\alpha} \in F$ and \defi{reduced trace} $\trd(\alpha)=\alpha+\overline{\alpha} \in F$ belong to $F$ for all $\alpha \in B$.  The \defi{discriminant} 
\[ \Delta(\alpha)=-(\alpha-\overline{\alpha})^2=4\nrd(\alpha)-\trd(\alpha)^2 \in F \]
defines a positive definite quadratic form $\Delta:B/F \to F$.

A \defi{right fractional ideal} of $\calO$ is a finitely generated $\Z_F$-submodule $I \subset B$ with $IF=B$ such that $I\calO \subseteq I$.  The \defi{right order} of a right fractional ideal $I$ is
\[ \calO_R(I)=\{x \in B : Ix \subseteq  I\}. \] 
Left fractional ideals and left orders are defined analogously.  A right fractional ideal of $\calO$ has $\calO_R(I)=\calO$ if and only if $I$ is \defi{(left) invertible}, so there exists a left fractional ideal $I^{-1}$ with $\calO_R(I^{-1})=\calO_L(I)$ such that $I^{-1}I=\calO$.

Let $I,J$ be invertible right fractional $\calO$-ideals.  We say that $I$ and $J$ are in the same \defi{right ideal class} (or are \defi{isomorphic}) if there exists an $\alpha \in B^\times$ such that $I=\alpha J$, or equivalently if $I$ and $J$ are isomorphic as right $\calO$-modules.  We write $[I]$ for the equivalence class of $I$ under this relation and denote the set of invertible right $\calO$-ideal classes by $\Cl \calO$.  The set $\Cl \calO$ is finite and $H=\#\Cl \calO$ is independent of the choice of Eichler order $\calO$.  
Let $I_1,\dotsc,I_H$ be a set of representatives for $\Cl\calO$
such that $\nrd I_i$ is coprime to $\frakN$ for all $i$.

Let $U_i$ be the space of spherical harmonic polynomials of homogeneous degree $(k_i-2)/2$ on $B_{v_i}/F_{v_i}\cong\R^3$ (with respect to $\Delta$) and let
\[
  W_k = \bigoplus_{i=1}^{n} U_i.
  \]
Then $B^\times$ (and hence $B^\times/F^\times$) acts on $W_k$ by conjugation on each factor.
The space $M_k^B(\frakN)$ of \defi{quaternionic modular forms} of \defi{weight} $k$ and \defi{level} $\frakN$ for $B$ can be represented as
\[
  M_k^B(\frakN) \cong \bigoplus_{i=1}^H W_k^{\Gamma_i}
\]
where $\Gamma_i = \calO_L(I_i)^\times/\Z_F^{\times}$.  
By the correspondence of Eichler, Shimizu, and Jacquet--Langlands,
$M_k^B(\frakN)$ is isomorphic as a Hecke module to a subspace of $\M_k(\frakN)$ of forms that are new at all the primes where $B$ is ramified.
In particular, $f$ corresponds to a quaternionic modular form given
by $(P^{(1)}, \dotsc, P^{(H)})$ with $P^{(i)}\in W_k^{\Gamma_i}$.
Algorithms to compute the coefficients $P^{(i)}$ are due to Socrates--Whitehouse \cite{MR2175121} and Demb\'el\'e \cite{MR2291849} and are surveyed in the work of Demb\'el\'e--Voight \cite{DembeleVoight}.

The form $g \in M_{(k+1)/2}(4\frakN)$ associated to $f$ by Conjecture
\ref{conj:Waldspurger} can be (conjecturally) computed as
\begin{equation}
  g(z) = \sum_{i=1}^H \frac{1}{\#\Gamma_i}\Theta(\Lambda^{(i)}, P^{(i)} ; z)
\end{equation}
where $\Lambda^{(i)}=\calO_L(I_i)/\Z_F$ has quadratic form $\Delta$
and spherical harmonic polynomial $P^{(i)}$.
Therefore, to compute the (non)vanishing twists $L(f,1/2,\chi_D)$,
encoded in the coefficients $c_{|D|}(g)$, we need only compute the theta series
$\Theta(\Lambda^{(i)},P^{(i)})$ with sufficiently many terms
and choose a unique representative for each fundamental discriminant.

\begin{remark} \label{rmk:constL0}
If we follow this construction with $L(f,1/2)=0$, then the corresponding form has $g=0$: see Gross \cite{Gross} and B\"ocherer--Schulze-Pillot \cite{BSP} for the case $F=\Q$
and Xue \cite{Xue} and Sirolli \cite{Sirolli} for general $F$.  On the other hand, 
one expects a nonzero $g$ in
Conjecture~\ref{conj:Waldspurger};
see Mao--Rodriguez-Villegas--Tornar\'ia \cite{MRVT} for an extension of this algorithm to compute $g$ in this case.
\end{remark}


\subsection*{Theta series}

To expand a theta series over $\Z_F$ such as in the previous section, we
consider the form over $\Z$ given by the trace, as follows.
Let $\Lambda$ be an integral $\Z_F$-lattice 
with a positive definite quadratic form $Q$ and a spherical polynomial $P$, and
let $\Theta(\Lambda,P)$ be its theta series as in \eqref{eqn:ThetaLambdaP}.
We will exhibit a method to compute the expansion of
this series over all terms $q^\nu$ with $\Tr \nu \leq T$.

For each embedding $v_i:F \hookrightarrow \R$
the quadratic form $Q:\Lambda \to \Z_F$
yields a positive definite quadratic form $Q_i:\Lambda \to \R$,
and the sum
\begin{align*} 
\Tr Q : \Lambda \cong \Z^{dn} &\to \Z \\
x &\mapsto \Tr Q(x)
\end{align*}
is a positive definite quadratic form over $\Z$.  Then the Fincke--Pohst
algorithm \cite{MR777278}, based on the LLL lattice reduction algorithm
\cite{MR682664}, can be used to iterate over the elements of
\[
  X_T = \{x \in \Lambda : \Tr Q(x)\leq T\}
\]
using $O({T}^{dn/2})$ bit operations for fixed $\Lambda$, $Q$, $P$ as $T \to \infty$,
which is proportional to the volume of the associated region.
This allows one to compute the coefficients of 
\[ \Theta(\Lambda,P;z) = c_0 + \sum_{\nu \in (\Z_F)_+} q^{\Tr(\nu z)} \]
for all $\nu$ with $\Tr \nu\leq T$ as
\[
  c_{\nu} = \sum_{\substack{x\in X_T \\ Q(x)=\nu}} P(x) \,.
\]

\subsection*{Discriminants}

To test the conjectures highlighted in the introduction, among the coefficients in the theta series computed in the previous subsection we need to find a unique set of representatives for fundamental discriminants with bounded norm (not trace).  In this subsection, we explain how to do this.

First, we will need to work with ``balanced'' representatives of fundamental
discriminants, up to the action of $U_F=\Z_F^{\times 2}$.
To do so, we find a fundamental domain $\Delta$ for the action of $U_F$ on $F_+
\hookrightarrow \R_{>0}^n$, and choose the representative $D$ such that $-D$
lies in this fundamental domain.
The study of such domains was pioneered by Shintani \cite{MR0427231};
see also the exposition by Neukirch \cite[\S VII.9]{MR1697859}.

\begin{theorem}
There exists a $\Q$-rational polyhedral cone $\Delta$ that is a finite disjoint union of simplicial cones such that 
\[ \R_{>0}^n = \bigsqcup_{u \in U_F} u \Delta. \]
\end{theorem}

Shintani gives an effective procedure for computing $\Delta$, but it is quite complicated to carry out in practice.  Further algorithms for computing the cone $\Delta$ are given by Okazaki \cite{MR1251215} and Halbritter--Pohst \cite{MR1823195}.  

For concreteness, in this subsection we assume that $F$ is a real quadratic field.  (Similar arguments work for higher degree fields, including work done for real cubic fields \cite{MR565136,MR2922336}, but for degree $n=[F:\Q] \geq 4$, they are quite a bit more complicated.)  In this case, we have $\Z_F^{\times 2}=\eps^{\Z}$; replacing $\eps$ by $1/\eps$, we may assume that $\eps_2>\eps_1$, and then $\eps$ is unique.  Then the Shintani domain $\Delta$ is the cone over the vectors $(1,1)$ and $(\eps_1,\eps_2)$, so $(x_1,x_2) \in \Delta$ if and only if $1 \leq x_2/x_1 < \eps_2/\eps_1$.  

\begin{lemma} \label{lem:x1x2N}
If $x=(x_1,x_2) \in \Delta$ has $\N(x)=x_1x_2 \leq X$ then 
\[ \Tr x=x_1+x_2 \leq \Tr(\eps)\sqrt{X}. \]
\end{lemma}

\begin{proof}
By homogeneity, it suffices to prove this for the case $X=1$.  Then the maximum value of $\Tr x$ subject to $\N(x)=1$ and $x \in \Delta$ occurs at the vertex $(\eps_1,\eps_2)$.
\end{proof}

With Lemma \ref{lem:x1x2N} in hand, to compute the coefficients $c_{|D|}(g)$, one for each $D \in \calD(\Z_F;X)$, we apply the method of the previous subsection with $T=\gamma_F \sqrt{X}$ where $\gamma_F=\Tr(\eps)$ and select only those with $-D \in \Delta$.  We can further restrict to fundamental discriminants by factoring $D\Z_F$ to ensure that its odd part is squarefree, the exponent of an even prime is $\leq 3$, and finally that $D/e^2$ is not a square modulo $4$ for all even nonassociate prime elements $e$.  In a similar way, we can enumerate \emph{all} permitted $D \in \calD(\Z_F;X)$ using the unary theta series $Q(x)=x^2$ on $\Z_F$.  

\begin{remark} \label{rmk:tracenormslice}
For a totally real field $F$ of degree $n=[F:\Q]$, by similar reasoning and the arithmetic-geometric mean, there exists $\gamma_F>n$ depending only on $F$ such that for all $x \in \Delta$, we have
\[ n\N(x)^{1/n} \leq \Tr(x) \leq \gamma_F \N(x)^{1/n}, \]
i.e., $\Tr(x) =\Theta(\N(x)^{1/n})$.  (In fact, $\gamma_F=\max_\eps |\Tr(\eps)|$ for $\eps$ a ray of the Shintani cone.)  So for the purposes of testing our conjectures, one can use either trace or norm, and for higher degree fields, the former is much simpler to implement.
\end{remark}

\begin{remark}
The algorithms above are used for experimental purposes, so we do not give a
precise running time for them; however, we can give a rough idea of the running
time (which bears out in practice): for a fixed field $F$, the time to compute
$\{c_{|D|}(g) \;:\; D \in \calD(\Z_F;X)\}$ for a Hilbert modular newform $f \in
S_k(\frakN)$ over $F$ is governed by computing theta series, which is roughly
\[ O((\N\frakN)(X^{1/n})^{3n/2})=O((\N\frakN)\cdot X^{3/2}) \]
because $\Tr(D)=O(X^{1/n})$ for $-D\in\Delta$ and $\dim S_k(\frakN)=O(\#\Cl \calO)=O(\N\frakN)$ by Eichler's mass formula \cite[Corollaire V.2.5]{MR580949}.
\end{remark}

\subsection*{Example}

The Hilbert modular form of parallel weight $2$ and smallest level
norm over $F=\Q(\sqrt{5})$ is the form $f$ whose label is
\texttt{2.2.5.1-31.1-a} (see Section~\ref{sec:comps} for more on the
labeling system), and we take this form as an example to illustrate the above algorithms.  Let $w=(1+\sqrt{5})/2$, so $w$ satisfies $w^2-w-1=0$ and $\Z_F=\Z[w]$.  Let $\frakN=(5w-2)$, so $\N\frakN = 31$.  Then there is a unique form in $S_2(\frakN)$; it has rational Hecke eigenvalues 
\[ a_{(2)}=-3, a_{(2w-1)}=-2, a_{(3)}=2, a_{(3w-2)}=4, \dots \]
and is neither CM nor base change.  By the Eichler--Shimura construction, this form corresponds to the isogeny class of the elliptic curve
\begin{equation} \label{eqn:E5}
 E:y^2 + xy + wy = x^3 + (w + 1)x^2 + wx 
 \end{equation}
via $\#E(\Z_F/\frakp) = \N\frakp+1-a_\frakp$ for $\frakp \neq \frakN$.

The $L$-function $L(f,s)$ has conductor $Q=5^2 \cdot 31=775$ and $L(f,1/2)=0.3599289\ldots \neq 0$; the Atkin--Lehner eigenvalue is $w_{\frakN,f}=w_f=1$.  Like any form of parallel even weight $k$ over a real quadratic field, the form $f$ satisfies the parity condition $2k \equiv 0 \pmod{4}$.  

We find a quaternion algebra ramified only at the 2 real places of $F$ to be $B=\quat{-1,-1}{F}$, with Eichler order of level $\frakN$ given by
\[ \calO=\Z_F \oplus \Z_F(5w-2)i \oplus \Z_F\frac{(w+1)+(w+8)i+j}{2} \oplus \Z_F\frac{w+(w+25)i+k}{2}. \]
There are two right ideal classes in $\calO$, with the class $[I_1]=[\calO]$ and the nontrivial class represented by the ideal
\[ I_2 = (w + 3)\calO + \frac{(w+11)+(w+70)+j}{2}\calO \]
of reduced norm $\nrd(I)=(3w-2)$ itself of norm $11$.  We have $\#\Gamma_1=\#(\calO_1^{\times}/\Z_F^{\times})=5$ and $\#\Gamma_2=3$.  In parallel weight $2$, the weight space $W_k=\C$ is trivial, and we find that accordingly $M_2^B(\frakN)$ has dimension $2$ as a $\C$-vector space.  The eigenform $f$ corresponds to the eigenvector $5[I_1] - 3[I_2]$, so we have $P^{(1)}=5$ and $P^{(2)}=-3$.

We have $\calO_L(\calO)=\calO$, and the discriminant form $\Delta$ on $\Lambda^{(1)}=\calO/\Z_F$ gives a theta series
\begin{align*} 
\Theta(\Lambda^{(1)};z) &= \sum_{v \in \Lambda^{(1)}} q^{\Tr(\Delta(v) z)}  \\
&= 1 + q^{\Tr((-w+3)z} + q^{\Tr((w+2)z} + 0q^{\Tr(3z)} + q^{\Tr((4w+3)z)} + \dots
\end{align*} 
The terms $\nu=0,-4w+12,4w+8$ correspond to nonfundamental discriminants.  In a similar way we compute $\Theta(\Lambda^{(1)};z)$ and then compute 
\begin{align*}
g(z) &= \Theta(\Lambda^{(1)},P^{(1)};z) + \Theta(\Lambda^{(2)},P^{(2)};z)  = \frac{5}{5}\Theta(\Lambda^{(1)};z) - \frac{3}{3}\Theta(\Lambda^{(2)};z) \\
&= q^{\Tr((w+2)z)} + q^{\Tr((-w+3)z)} - q^{\Tr(3z)} - q^{\Tr((3w+3)z)} + q^{\Tr((-3w+6)z)} + q^{\Tr((4w+3)z)} + \dots
\end{align*}
We confirm that the quadratic twists of $E$ by the discriminants $D=-w-2,w-3,-3,-3w-3,3w-6,-4w-3$ all have rank $0$, consistent with this series; we find out first vanishings at $D=8w-43,-9w-38$, and the corresponding twists indeed have rank $2$.

\section{Conjectures about central values of $L$-functions}\label{sec:rmt}


Random matrix theory has proved useful in refining
conjectures related to the low-lying zeros of $L$-functions
\cite{KeatingSnaith1,KeatingSnaith2}.  In work of Conrey--Keating--Rubinstein--Snaith \cite{CKRS}, the following basic
question was considered.  Let $f \in S_k(N)$ be a newform with rational integer
coefficients.  For how many fundamental discriminants $D$ with $|D|\leq X$ does the twisted $L$-function 
$L (f,s,\chi_D)$ vanish at the center of the critical strip?  In a collection of papers \cite{MRVT,PT,TR}, a
number of variants of this problem were considered: the
weight of $f$ was allowed to vary, the level of $f$ was allowed to be
composite, and so on.

\subsection*{Simple heuristics}

To motivate the power of $X$ in our Conjectures~\ref{conj:ZF} and \ref{conj:Z}, we start with two simple heuristics that follow a similar coarse reasoning and ignore logarithmic factors.  Let
\[ \calN_f(X) = \#\bigl\{D \in \calD(\Z;X) \;:\;
  \text{$\chi_D(N)=-w_f$ and $L(f,1/2,\chi_D)=0$}\bigr\} \]
count the number of vanishing twists of $f$ with sign $+1$.  
If $k=2$, then $f$ corresponds to an elliptic curve $E$ over $\Q$ of conductor $N$, and so assuming the (weak) Birch--Swinnerton-Dyer conjecture, the function $N_f(X)$ counts the proportion of twists of $E$ with rank at least two.  
The function $\calN_f(X)$ has a heuristic estimate attributed to Sarnak, as follows.  As discussed in Section 2, Conjecture~\ref{conj:Waldspurger} relates the central value of the $D$-th twist of $f$ to the
coefficient $c_{|D|}(g)$ of a particular half-integer weight modular
form $g$ attached to $f$ according to the formula
\begin{equation}\label{eqn:WaldspurgerQ}
  L(f,1/2,\chi_D)=\kappa_f \frac{c_{|D|}(g)^2}{\sqrt{\abs{D}}}.
\end{equation}
The Ramanujan--Petersson bound on the coefficients of $g$ is $\abs{D}^{1/4+\varepsilon}$ for all $\varepsilon>0$, and so, if $c_{|D|}(g)$ takes a value in the range $\bigl[0,\abs{D}^{1/4}\bigr]$ for $\abs{D}\leq X$ that is not too far from uniform, then it should take the value $0$ approximately $X^{3/4+\varepsilon}$ of the time.  

The same heuristic extends to higher weight and to the case of Hilbert modular forms, as follows.  Let $f$ be a Hilbert modular form of parallel weight $k$ over a field $F$ of degree $n=[F:\Q]$.  The
Lindel\"of Hypothesis (see e.g.\ Baruch--Mao \cite[Conjecture 1.4]{BaruchMao}) implies that the central values of any family of quadratic twist $L$-functions satisfying the Riemann Hypothesis have
\begin{equation} \label{eqn:Lindelof}
  L(f,1/2,\chi_D) = O(\abs{\N(D)}^\varepsilon)
\end{equation}
for all $\varepsilon>0$ as $\abs{\N(D)} \to \infty$.  By Conjecture~\ref{conj:Waldspurger}, the bound \eqref{eqn:Lindelof} is equivalent to
\[ \abs{c_{|D|}(g)}=O(|\N(D)|^{(k-1)/4+\varepsilon}) \]
for all $\varepsilon>0$, so this would also be asserted by the Ramanujan--Petersson bound for half-integral weight Hilbert modular forms.

Now, consider the context of Conjecture~\ref{conj:ZF}.  If for each $D \in \calD(\Z_F;X)$ the coefficient $c_{|D|}(g)$ has a distribution in $[0,|\N(D)|^{(k-1)/4}]$ that is not too far from uniform, then the expected number of times when the value is $0$ is approximately $X^{1-(k-1)/4+\varepsilon}$, since $\calD(\Z_F;X)=\Theta(X)$.  We note that this bound is independent of $n=[F:\Q]$.

On the other hand, consider the case considered in Conjecture~\ref{conj:Z} where we twist only by
twist $D \in \calD(\Z;X)$.  If $c_{|D|}(g)$ is still distributed in a way not too far from uniform, then since
\[
|\N(D)|^{(k-1)/4+\varepsilon}=|D|^{n(k-1)/4+n\varepsilon}
\]
then it should vanish about $X^{1-n(k-1)/4+\varepsilon}$ for all $\varepsilon>0$.

\subsection*{The power of $\log$}

We now pay attention to the power of $\log X$ appearing in our conjectures, and for this refinement we must dig a bit deeper.  Consider first the case $F=\Q$ and weight $k$.  
We follow the heuristic \cite{CKRS,KeatingSnaith2} prescribed by random matrix theory, verified by the large scale computation of central values of $L$-functions twisted by a quadratic character over $\Q$; see also further work by Watkins \cite{MR2523320}.  In this theory, low-lying zeros of $L$-functions are related to values of characteristic polynomials of random matrices of $SO(2m)$ and we deduce:
\begin{equation}\label{eqn:prob2}
\textrm{Prob}[L(f,1/2,\chi_D)=0]\sim C_f 2^{\lambda(D,f)}\frac{(\log |D|)^{3/8}}{\abs{D}^{(k-1)/4}}
\end{equation}
for a constant $C_f$.  Here $2^{\lambda(D,f)}$ is a multiplicative function that increases with the number of prime divisors of $D$ and depends on the arithmetic and geometric nature of $f$.  We assume, too, that for some constant $b_f$, we have $2^{\lambda(D,f)}\approx (\log D)^{b_f}$ on average.  For example, Delaunay--Watkins \cite{MR2322344} for the case of an elliptic curve $E$ (so $F=\Q$ and $k=2$) predict the exponent $b_E$ where $b'_E=1,\sqrt{2}/2,1/3,\sqrt{2}/2-1/3$ depending on properties of the mod $2$ Galois representations occurring for elliptic curves in the isogeny class of $E$ and, roughly speaking, define $2^{\lambda(D,f)}$ to be the product of the Tamagawa numbers $g_p$ for $p\mid D$.  A further refinement could be achieved by considering $2^{\lambda(D,f)}$ to be related to the power of 2 that divides the coefficient $c_{|D|}(g)$ of the half integer weight form that corresponds to $f$ as in Conjecture~\ref{conj:Waldspurger} over $\Q$ where it is a theorem. To find the total number of vanishings, we sum the probability in \eqref{eqn:prob2} over all discriminants up to $X$.

In the setting of Hilbert modular forms we arrive at a formula in an identical way.  As far as we can tell, the random matrix theory arguments do not see the field over which the Hilbert modular form is defined---they only care about the $L$-function and its functional equation.  It is reasonable then to suppose that we have 
\begin{equation}\label{eqn:prob2F}
\textrm{Prob}[L(f,1/2,\chi_D)=0]\sim C_f 2^{\lambda_F(D,f)}\frac{(\log \abs{\N(D)})^{A_n}}{\abs{D}^{(k-1)/4}}
\end{equation}
in analogy with \eqref{eqn:prob2}, where $A_n$ is a constant that depends only on the degree of the field (arising from the normalization of $m$ in $SO(2m)$; $A_1=3/8$) and $2^{\lambda_F(D,f)}$ be a multiplicative function that increases with the number of primes in $\Z_F$ dividing the ideal~$(D)$.
For Conjecture~\ref{conj:ZF}, we suppose that $2^{\lambda_F(D,f)}$ for $D\in\mathcal{D}(\Z_F)$ is, for some $b_f$, on average roughly of size $(\log \N(D))^{b_f}$ and so we have
\begin{align*}
  \calN_f(\Z_F;X)
  &\sim
  \sum_{D\in\calD(\Z_F;X)} 2^{\lambda_F(D,f)}\frac{(\log \abs{\N(D)})^{A_n}}{\abs{\N(D)}^{(k-1)/4}}  \\
  &\sim  \sum_{D\in\calD(\Z_F;X)}(\log \abs{\N(D)})^{b_f} \frac{(\log \abs{\N(D)})^{A_n}}{\abs{\N(D)}^{(k-1)/4}}
  \sim C_f\, X^{(1-(k-1)/4)}(\log X)^{b_f+A_n}.
\end{align*}
For example, if $\lambda_F(D,f)=\omega_F(D)$, we have $b_f=1$: the congruent number curve for $F=\Q$ is an example where $\lambda_{\Q}(D,f)=\omega(D)$.

Now consider the sum
\begin{equation}\label{eqn:sum}
  \sum_{D\in\calD(Z;X)} 2^{\lambda_F(D,f)} \frac{(\log \abs{\N(D)})^{A_n}}{\abs{\N(D)}^{(k-1)/4}},
\end{equation}
as required by Conjecture~\ref{conj:Z}.  Now we must consider $\lambda_F(D)$ with $D\in\Z$.  Let $e_F$ be the expected number of distinct primes $\frakp$ in $\Z_F$ dividing $(p)$ for $p \in \Z$ prime.  For example, for $F$ quadratic we have $e_F=2(1/2)+1(1/2)=3/2$, and for $F$ of degree $n$ with $\Gal(F/\Q) \cong S_n$, we have $e_F=1+1/2+\dots+1/n$ is the $n$th harmonic number.  Then $2^{\lambda_F(D,f)}\sim 2^{e_F\lambda_{\Q}(D,f)} \approx \log(D)^{b_f\,e_F}$ where the $b_f$ is as in the beginning of this section.  Substituting in this estimate, and using that $\N(D) = D^n$ for $D \in \Z$, we are led to conjecture that
\[
  \calN_f(\Z;X)\sim
  C_{f,\Z} X^{(1-n(k-1)/4)}(\log X)^{b_f\,e_F + A_n},
\]
as predicted by Conjecture~\ref{conj:Z}, with $b_{f,\Z}=b_f\,e_F$.

\section{Computations}\label{sec:comps}

In this section we describe the results of the computations resulting from our implementation of the algorithms in Section~\ref{sec:algs}.  When describing a Hilbert modular form we use the labelling found in the LMFDB \cite{lmfdb}.  For example, the label \texttt{2.2.5.1-[2,4]-37.1-b} means the form over the field $F$ with label \texttt{2.2.5.1} (``degree $2$ with $2$ real places of discriminant $5$ numbered $1$'', i.e., $F=\Q(\sqrt{5})$) with weight $k=(2,4)$, level norm $31$ and ideal numbered $1$, and isogeny class $b$.  If $k=(2,\dots,2)$ is parallel weight $2$, we will suppress it and write simply \texttt{2.2.5.1-31.1-a}, for example.

\subsection*{Verification}

To make sure that we got statement of Conjecture~\ref{conj:Waldspurger} exactly correct in the cases we
care about, we computationally verified the
statement as written by comparing the central values we get by
computing them directly using \textsf{lcalc} \cite{lcalc} and by using
the coefficients of the theta series computed with the algorithms
described below.   Consider the ratio
 $D/3\cdot L(f,1/2,\chi_D)/L(f,1/2,\chi_{-3})$ where $D\in\Z$.  By
 Conjecture~\ref{conj:Waldspurger}, this ratio should be the square of an
 integer, namely the square of be the ratio of the corresponding
 coefficients of the associated half-integral modular form.  For the forms \texttt{2.2.5.1-31.1-a} and \texttt{2.2.13.1-4.1-a}, we checked the twists $D=-3,-7,-11,-19,-23,-43$.  
 
 We also verified for the form \texttt{2.2.5.1-31.1-a} and all $\abs{\N(D)} < 10000$ that the vanishing of $c_{|D|}$ matched the vanishing of $L(E,1,\chi_D)$ with $E$ as in \eqref{eqn:E5} using \textsf{Magma}.  (Similar verification was performed by Sirolli \cite{Sirolli}.)
 
\subsection*{Overview of examples}

A Hilbert modular form $f$ over $F$ is a \defi{base change} (BC) if there exists a proper subfield $F' \subsetneq F$ and a Hilbert modular form $f'$ over $F'$ of some weight and level such that the base change of $f'$ (when it exists) to $F$ is equal to $f$ (up to quadratic twist). 

For each of the examples below, we tabulated data related to vanishings of central values of twists and data related to the distributions of values of coefficients.
\begin{enumerate}
\item $F$ quadratic, $k=2$
 (\texttt{2.2.5.1-31.1-a}, neither CM nor BC;
 \texttt{2.2.8.1-9.1-a}, BC but not CM;
 \texttt{2.2.24.1-1.1-a}, BC and CM);
\item $F=\Q(\sqrt{5})$, $k=4$
  (\texttt{2.2.5.1-[4,4]-11.1-a}, neither BC nor CM); and 
\item $F$ cubic, $k=2$
  (\texttt{3.3.49.1-41.1-a}, neither BC nor CM).
\end{enumerate}
(We did not collect enough data for a base change and CM form; a potential candidate is the form \texttt{2.2.24.1-1.1-a}.)

\subsection*{A first example}

We dedicate particular attention to the data related to the form \texttt{2.2.5.1-31.1-a} which we gave as an example at the end of Section \ref{sec:algs}.  In Table~\ref{tbl:obligatory}, we give some timings to provide some idea of how long our computations of this example took and also to provide data related to Conjectures~\ref{conj:ZF} and \ref{conj:Z}.  We study the vanishings $\calN_f(\Z_F;X)$ and $\calN_f(\Z;X)$ as defined in \eqref{eqn:NfF} and \eqref{eqn:NfZF}.
See also Figure~\ref{fig:obligatory} for a graphical representation of these data as well as some data related to the distributions of the Fourier coefficients of the Shimura lift of $f$.

\begin{table}[h]
\[ 
\begin{array}{c|c|cc|cc}
X & \text{time} & \#\mathcal{D}(\Z_F;X) & \calN_f(\Z_F;X) & \#\calD(\Z;X) & \calN_f(\Z;X) \\
\hline
10^2 & 0.07\text{s} & 3 & 0 & & \\ 
10^3 & 0.4\text{s} & 41 & 0&  & \\
10^4 & 5\text{s} & 439 & 41 &  & \\
10^5 & 1 \text{m} 40 \text{s} & 4481& 397 &  & \\
10^6 & 47 \text{m} & 44865 & 3173 &  & \\
10^7 & 46 \text{h} & 448667 & 24748 &387 & 50\\ 
10\cdot 10^7 & & 4486620 & 183100 & 1229 & 120\\
15\cdot 10^7 & 1.5 \text{y} & 6729969 & 259525 & 1509 & 141\\
\end{array}
\]
\caption{Statistics for the Hilbert modular form \texttt{2.2.5.1-31.1-a} over $\Q(\sqrt{5})$}\label{tbl:obligatory}
\end{table}

Because the power $b_f$ of $\log X$ is still uncertain, we appealed to another prediction made by random matrix theory (RMT): for a prime $\frakq \nmid \frakN$, we define
\begin{align*} 
&\calN_f(\Z_F,\frakq,\pm 1;X)=\#\biggl\{D \in \calD(\Z_F;X) \;:\;
    \text{$\chi_D(\frakN)=(-1)^n\,w_f$ and $L(f,1/2,\chi_D)=0$} \\
&\qquad\qquad\qquad\qquad\qquad\qquad\qquad\qquad\qquad\qquad \text{and $\legen{D}{\frakq} = \pm 1$} \biggr\}
\end{align*}
as well as
\begin{equation} \label{eq:Nfq}
R_\frakq(f;X)=\frac{\calN_f(\Z_F,\frakq,+1;X)}{\calN_f(\Z_F,\frakq,-1;X)}
\quad \text{and} \quad R_\frakq(f)=\lim_{X\to \infty} R_\frakq(f;X).
\end{equation}
Then one expects that the limit $R_\frakq(f)$ exists \cite[Conjecture 2]{CKRS} and
\[ R_\frakq(f) = \sqrt{\frac{N\frakq+1-a_\frakq}{N\frakq+1+a_\frakq}}. \]
In Table \ref{tbl:obligatory_cong}, we computed the ratios \eqref{eq:Nfq} for several primes; the agreement is good, which gives a general indication that the predictions of RMT are relevant in this context.

\begin{table}[h]
\[ 
\begin{array}{c|c|cc|cc}
N\frakq & R_\frakq(f;2\cdot 10^7) & \displaystyle{\sqrt{\frac{N\frakq+1-a_\frakq}{N\frakq+1+a_\frakq}}} \\
\hline
5 & 20925/14986 = 1.396 & 1.414 \\
9 & 18237/22481 = 0.811 & 0.817 \\
11 & 17293/24674 = 0.701 & 0.707 \\
11 & 24899/17140 = 1.453 & 1.414 \\
19 & 19955/24802 = 0.805 & 0.817 \\
19 & 24847/20015 = 1.241 & 1.225 \\
29 & 23955/22424 = 1.068 & 1.069 \\
29 & 24040/22294 = 1.078 & 1.069
\end{array}
\]
\caption{Congruence ratios for \texttt{2.2.5.1-31.1-a} over $\Q(\sqrt{5})$}\label{tbl:obligatory_cong}
\end{table}

\begin{remark}
We did one extra experiment with this example by taking a slice of coefficients farther out for the example .  In order to do this efficiently, we took a ``trace slice'', not a ``norm slice''; according to Remark \ref{rmk:tracenormslice}, we expect qualitatively the same behavior, but potentially a different constant.  We computed vanishing statistics for all negative fundamental discriminants $D$ with 
\[   3\sqrt{10^9} < \abs{\Tr(D)} \leq 3\sqrt{10^9+10^6} \]
(here $\gamma_F=3$ for $F=\Q(\sqrt{5})$); according to Remark \ref{rmk:tracenormslice}, this corresponds to the range with
\[   10^9 < \abs{\N(D)} \leq 10^9 + 10^6 \]
up to a constant factor.
We find that there are $1761$ vanishings out of $75688$ for twists over $\Z_F$ and $0$ vanishings out of $3$ for twists over $\Z$.  

We also computed the ratios \eqref{eq:Nfq} inside this slice: we also found weak agreement, e.g.,
\[ \frac{\calN_f(\Z_F;\frakq;+1;10^9+10^6)-\calN_f(\Z_F;\frakq;+1;10^9)}
{\calN_f(\Z_F;\frakq;-1;10^9+10^6)-\calN_f(\Z_F;\frakq;-1;10^9)} = \frac{822}{536} = 1.533 \approx 1.414 \]
for $\N(\frakq)=5$, and similarly for $\N(\frakq)=9,11,11$ we found $625/877=0.713 \approx 0.816$, $609/874 = 0.697 \approx 0.707$, and $903/597 = 1.513 \approx 1.414$.
\end{remark}

\subsection*{Vanishings}

 We mention further experimental results related to vanishings.  In
 Figures~\ref{fig:obligatory}--~\ref{fig:graphs_wt2_cubic} we provide a graphical representation of the ratio 
 $\calN_f(\Z_F,X) / (X^{3/4}(\log X)^{11/8})$ for a variety of Hilbert modular forms of weight 2.  These plots all qualitatively get flat at about the same
 rate, suggesting to us that they are obeying the same qualitative law, the law from Conjecture~\ref{conj:ZF}.
In Figure~\ref{fig:wt4} we provide a graphical representation of
$\calN_f(\Z_F,X)/(X^{1/4}(\log X)^{11/8})$ for a form of weight 4,
as well as some data related to the distribution of the Fourier coefficients of the Shimura lift of $f$.

Our experiments provide convincing evidence for Conjecture~\ref{conj:ZF}, but cannot help us find the power of log despite having a  massive amount of data; by comparison, we can provide almost no evidence for 
Conjecture~\ref{conj:Z} since we collected only a tiny amount of data.  On the other hand, the simple heuristics we
provided in Section~\ref{sec:rmt}, assuming that Conjecture~\ref{conj:ZF} is
correct, gives strong evidence for Conjecture~\ref{conj:Z}.  

\subsection*{Distribution of coefficients}

We computed and stored the coefficients of the Shimura lifts of the
examples mentioned above for two reasons.  First, while the data related to
vanishings discussed in the previous section grows on the order of
$X^{3/4}$, the number of coefficients grows on the order of $X$.  So,
we automatically have more data at our disposal to compare with the
predictions of random matrix theory \cite{CKRS,CKRShalfint}.
Second, the distributions of Fourier coefficients of modular forms of
even integer weight are related to the Sato-Tate Conjecture; given a modular form $f$ the
distribution of its Fourier coefficients depends on, for example,
whether or not the modular form is CM.  The distributions of the
Fourier coefficients of the half-integral weight forms mentioned above are also plotted in
Figures~\ref{fig:obligatory}--~\ref{fig:wt4}  and they all exhibit qualitatively the same behavior, independent of whether or not the form whose
Shimura lift we are considering is CM or not.  

The data we compute are consistent with the random matrix theory predictions as described in \cite{CKRShalfint} and the distributions we compute match the predictions well.

\begin{figure}
\includegraphics[scale=0.5]{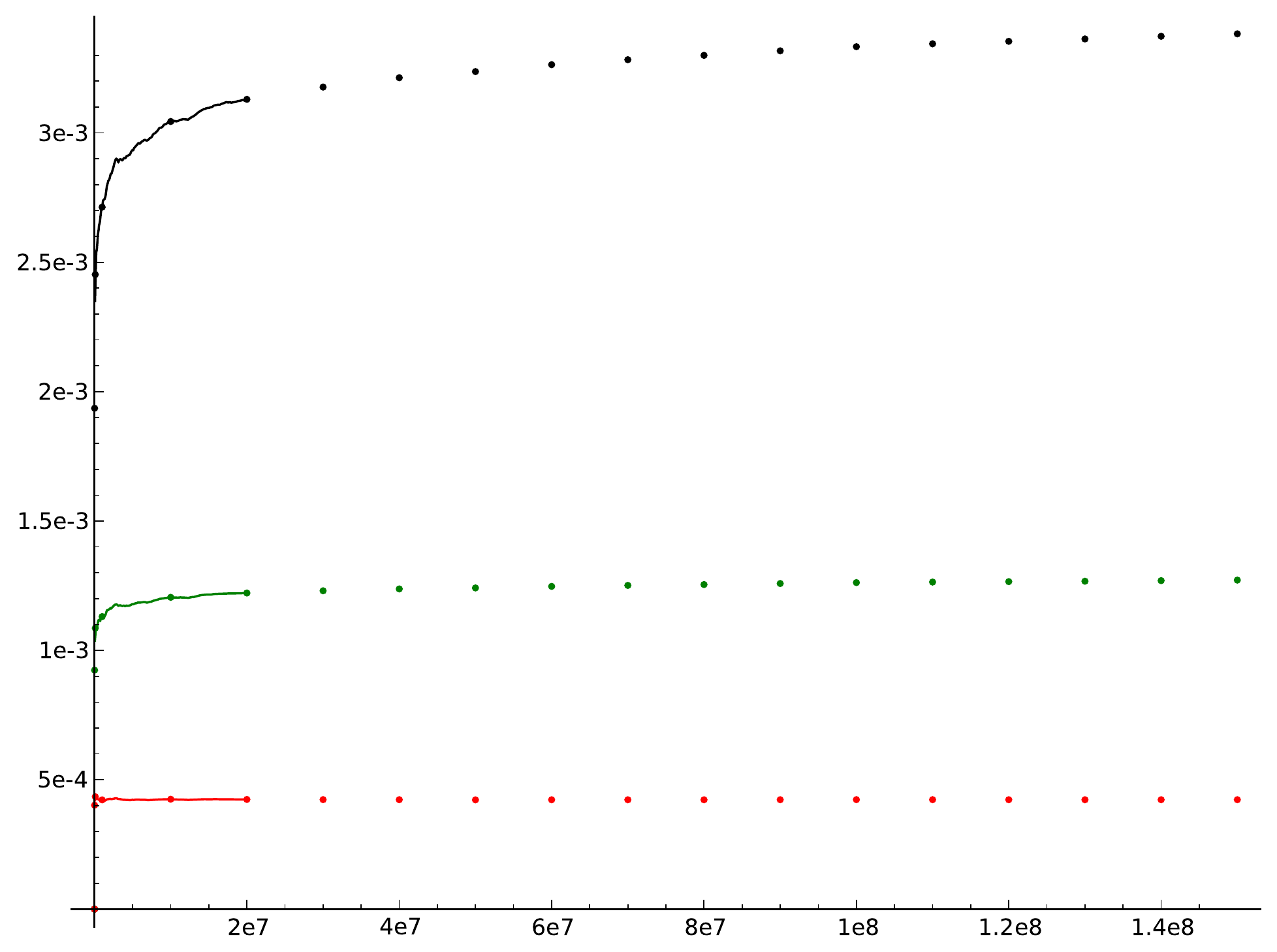}\hfill\includegraphics[scale=.5]{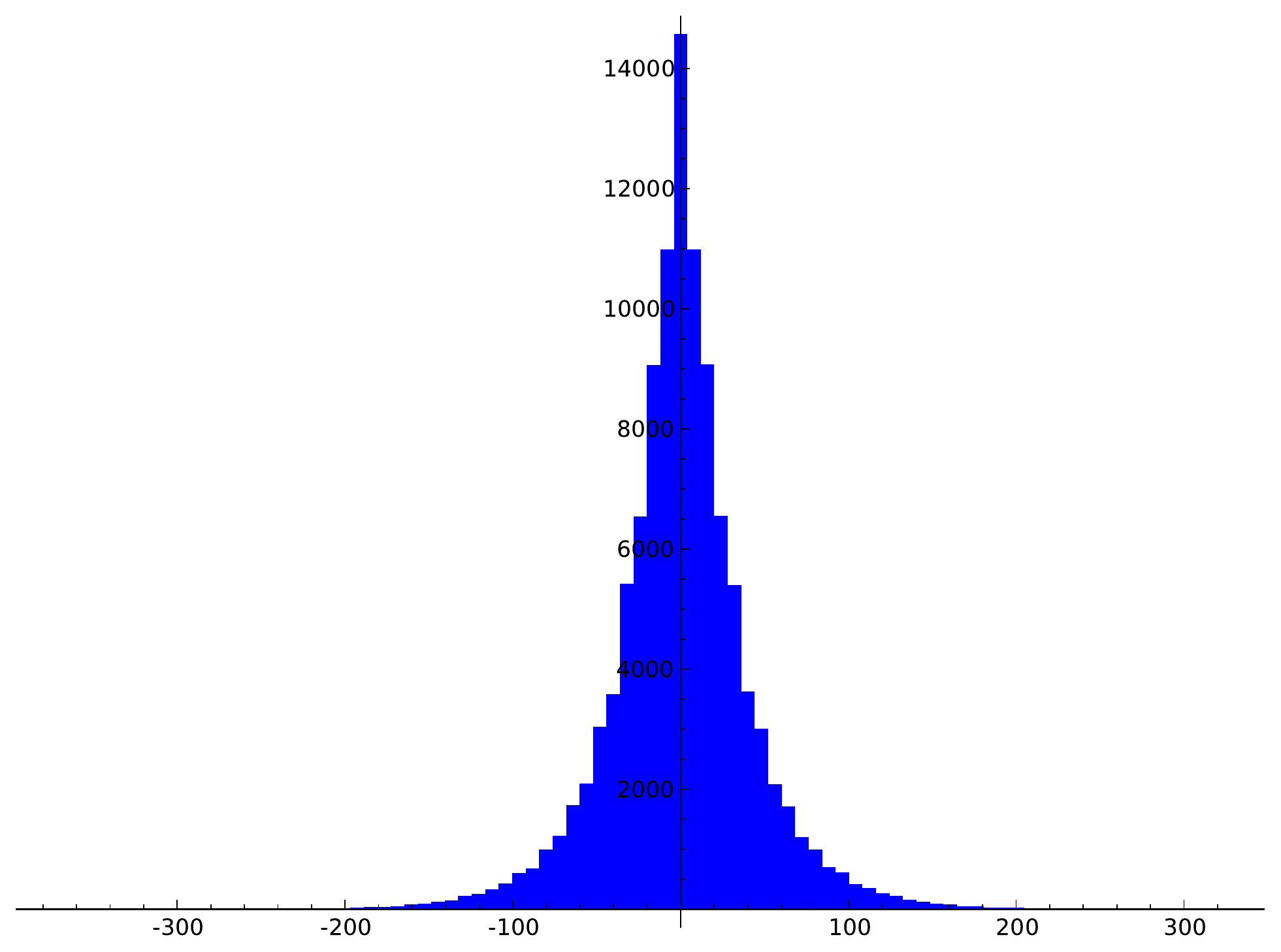}
\caption{The top figure is $\calN_f(\Z_F,X)/(X^{3/4}(\log X)^{e_f})$ for the form \texttt{2.2.5.1-31.1-a} and for various powers of log.  The powers of log are $e_f=11/8$ (shown in black), $e_f=41/24$ (shown in green) and $e_f=25/12$ (shown in red).  The powers of log were determined as follows: $e_f=11/8$ is the prediction analogous to the one made by Delaunay--Watkins for curves with full 2-torsion over $\Q$ but over $F$; $e_f=25/12$ is the least-squares fit of the data and $e_f=41/24$ comes from fitting the data to the computation of vanishings in a slice (see Remark~\ref{rmk:tracenormslice}) and comparing it with the vanishings up to $X$.  The bottom figure is the distribution of the values of the coefficients $c_{|D|}(g)$ that corresponds to $f$ as in Conjecture~\ref{conj:Waldspurger},
for $\N(D)<2\cdot10^7$.}
\label{fig:obligatory} 
\end{figure}

\begin{figure}[h]
  \includegraphics[scale=.3]{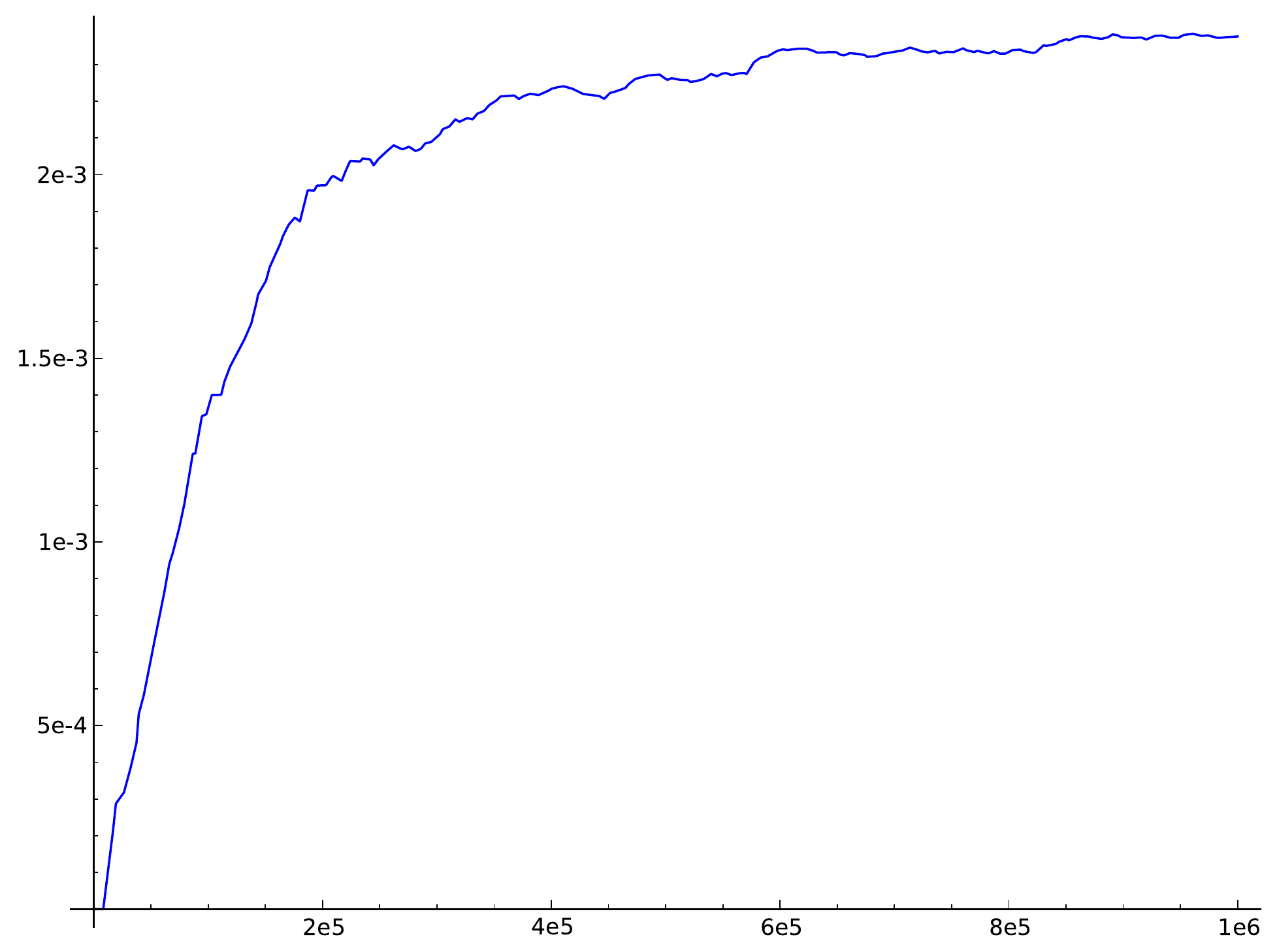}\hfill  \includegraphics[scale=.3]{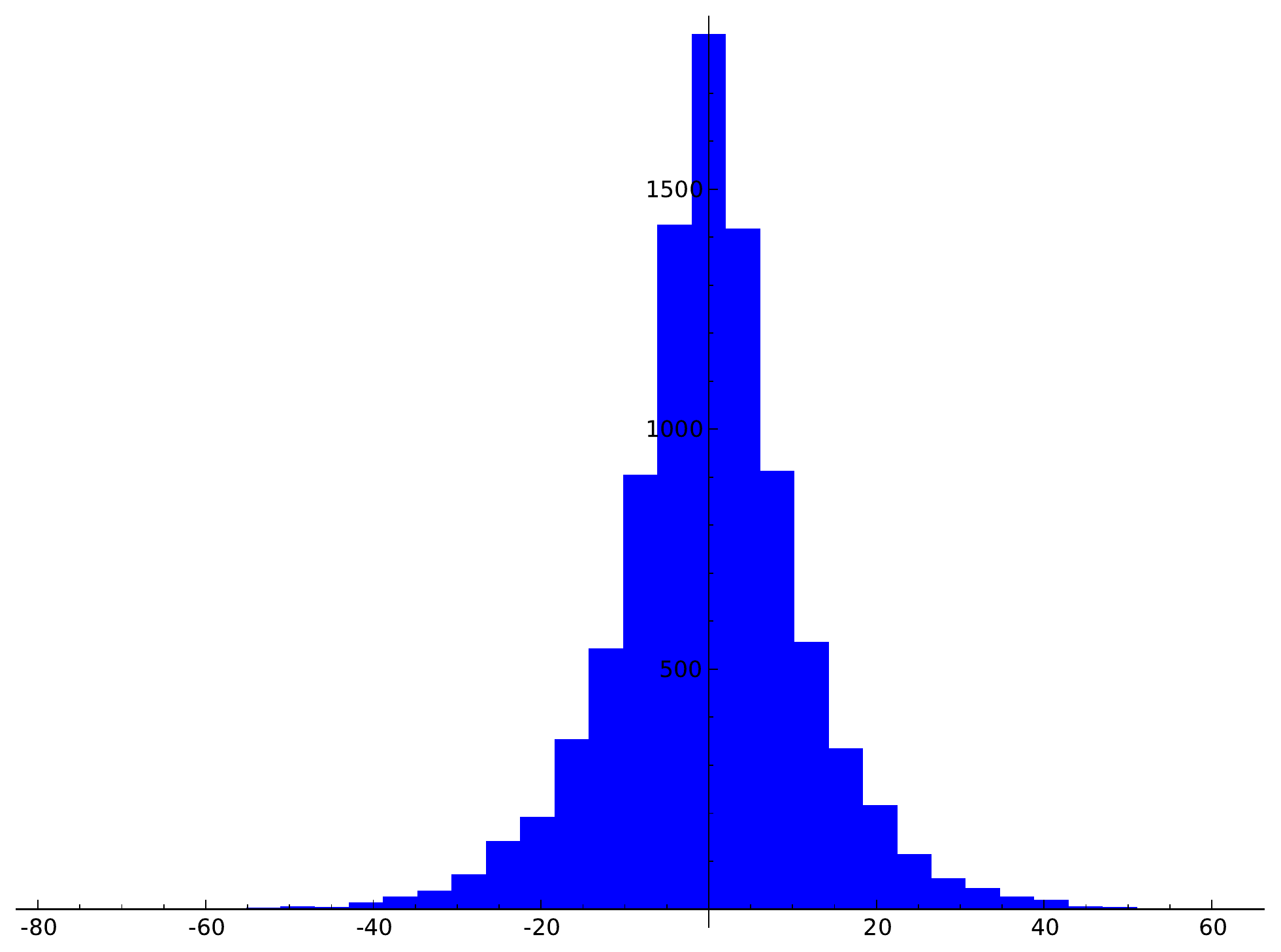}
\caption{$\#N_f(\Z_F,X) / (X^{3/4}(\log X)^{11/8})$ for the form \texttt{2.2.8.1-9.1-a} on the left and the distribution of the values of the coefficients $c_{|D|}(g)$ that corresponds to $f$ as in Conjecture~\ref{conj:Waldspurger}.}\label{fig:graphs_wt2_bc} 
  \end{figure}

\begin{figure}[h]
  \includegraphics[scale=.3]{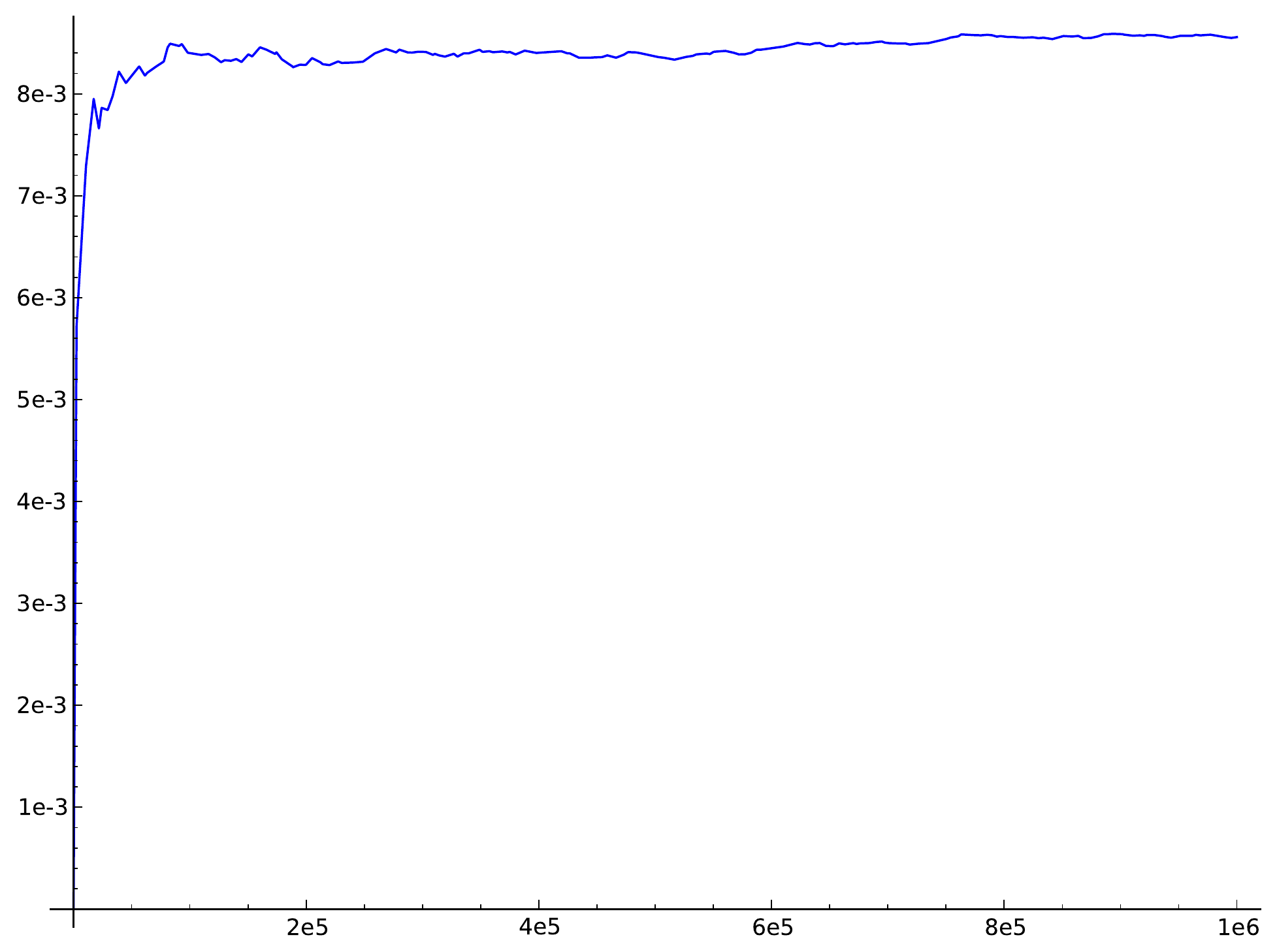}\hfill  \includegraphics[scale=.3]{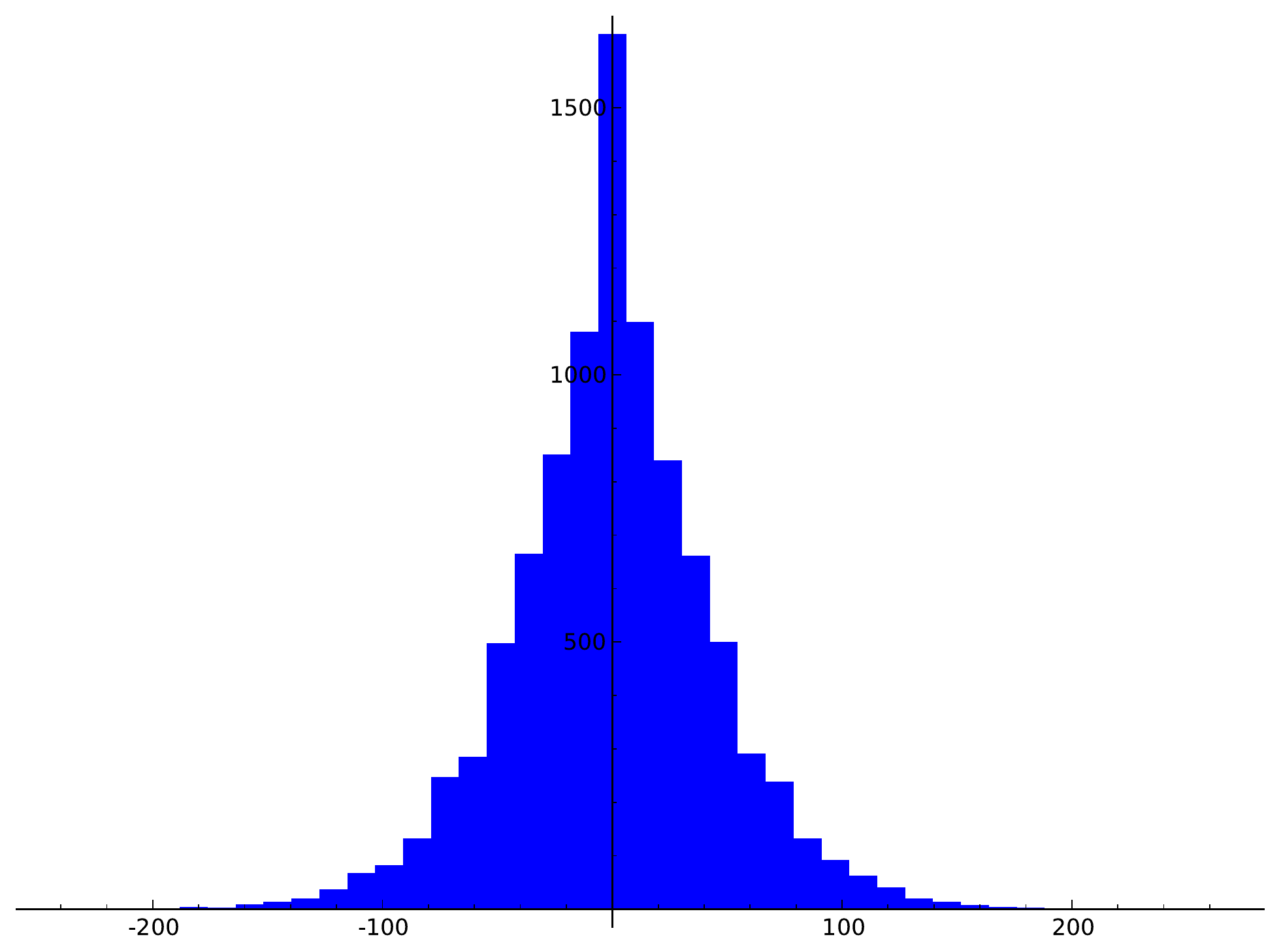}
\caption{$\#N_f(\Z_F,X) / (X^{3/4}(\log X)^{11/8})$ for the form  \texttt{2.2.24.1-1.1-a} on the left and the distribution of the values of the coefficients $c_{|D|}(g)$ that corresponds to $f$ as in Conjecture~\ref{conj:Waldspurger}.}\label{fig:graphs_wt2_cm} 
  \end{figure}
  
  \begin{figure}[h]
  \includegraphics[scale=.3]{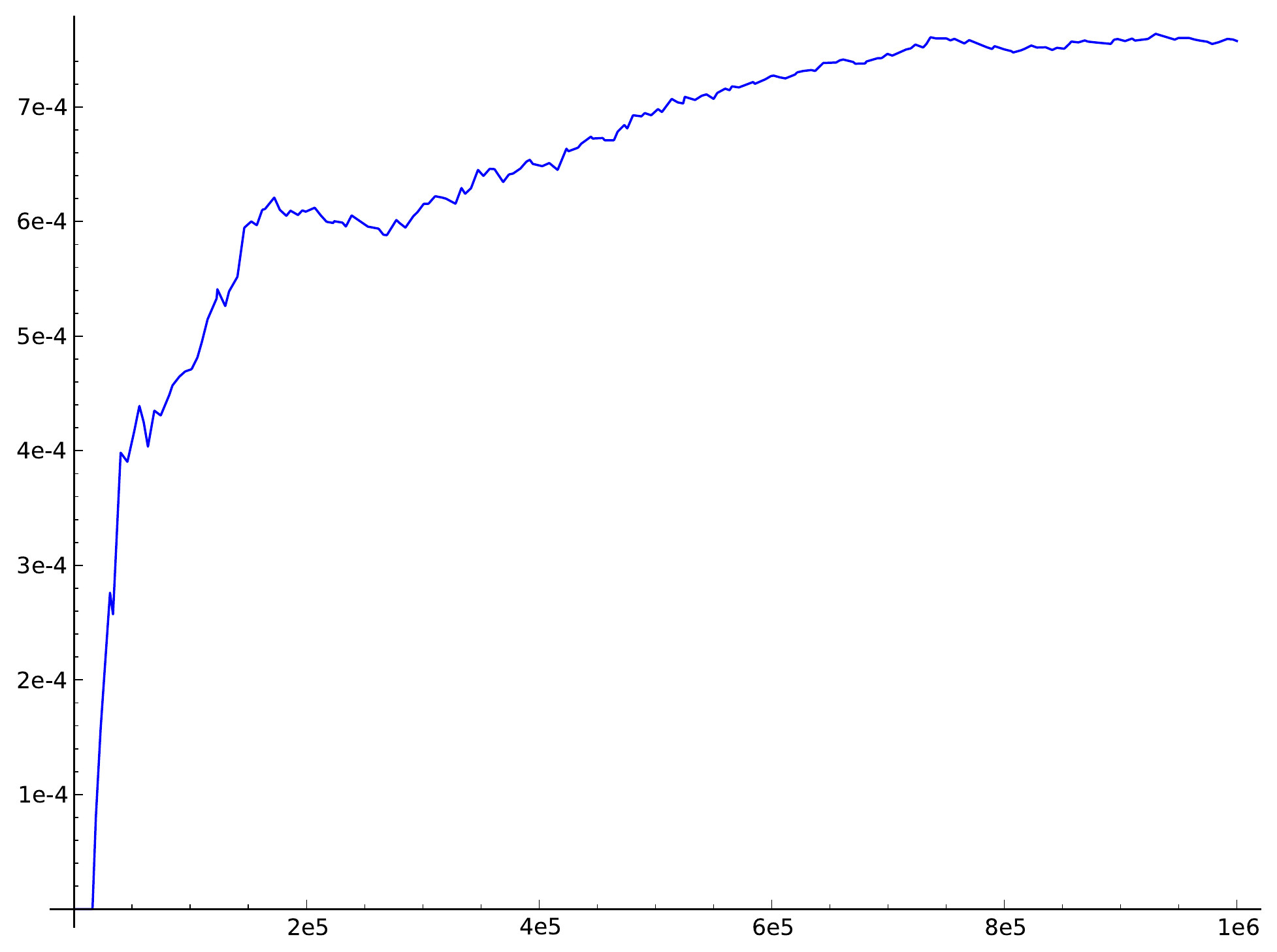}\hfill  \includegraphics[scale=.3]{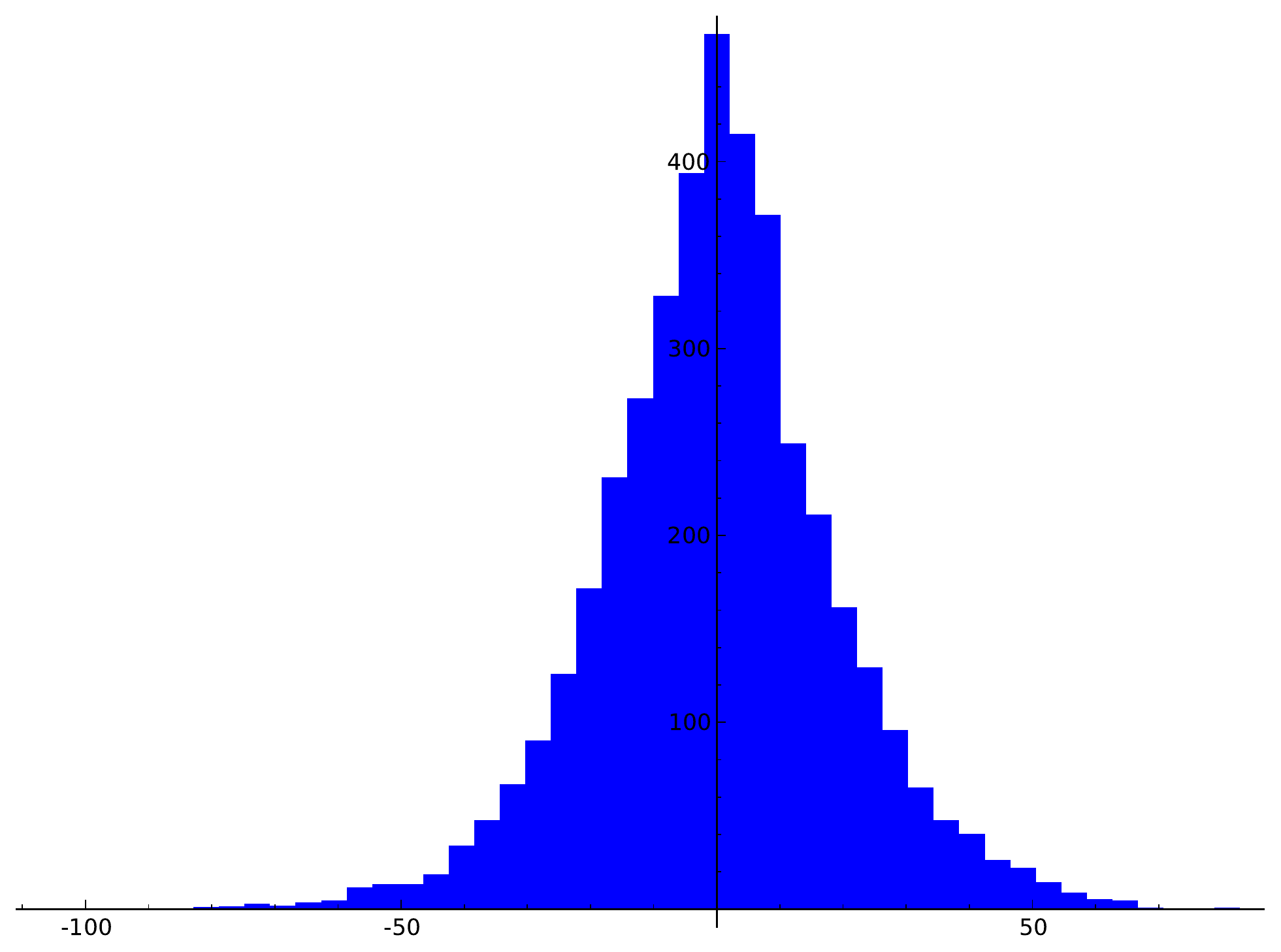}
\caption{$\#N_f(\Z_F,X) / (X^{3/4}(\log X)^{11/8})$ for the form  \texttt{3.3.49.1-41.1-a} on the left and the distribution of the values of the coefficients $c_{|D|}(g)$ that corresponds to $f$ as in Conjecture~\ref{conj:Waldspurger}.}\label{fig:graphs_wt2_cubic} 
  \end{figure}
  
    \begin{figure}[h]
  \includegraphics[scale=.3]{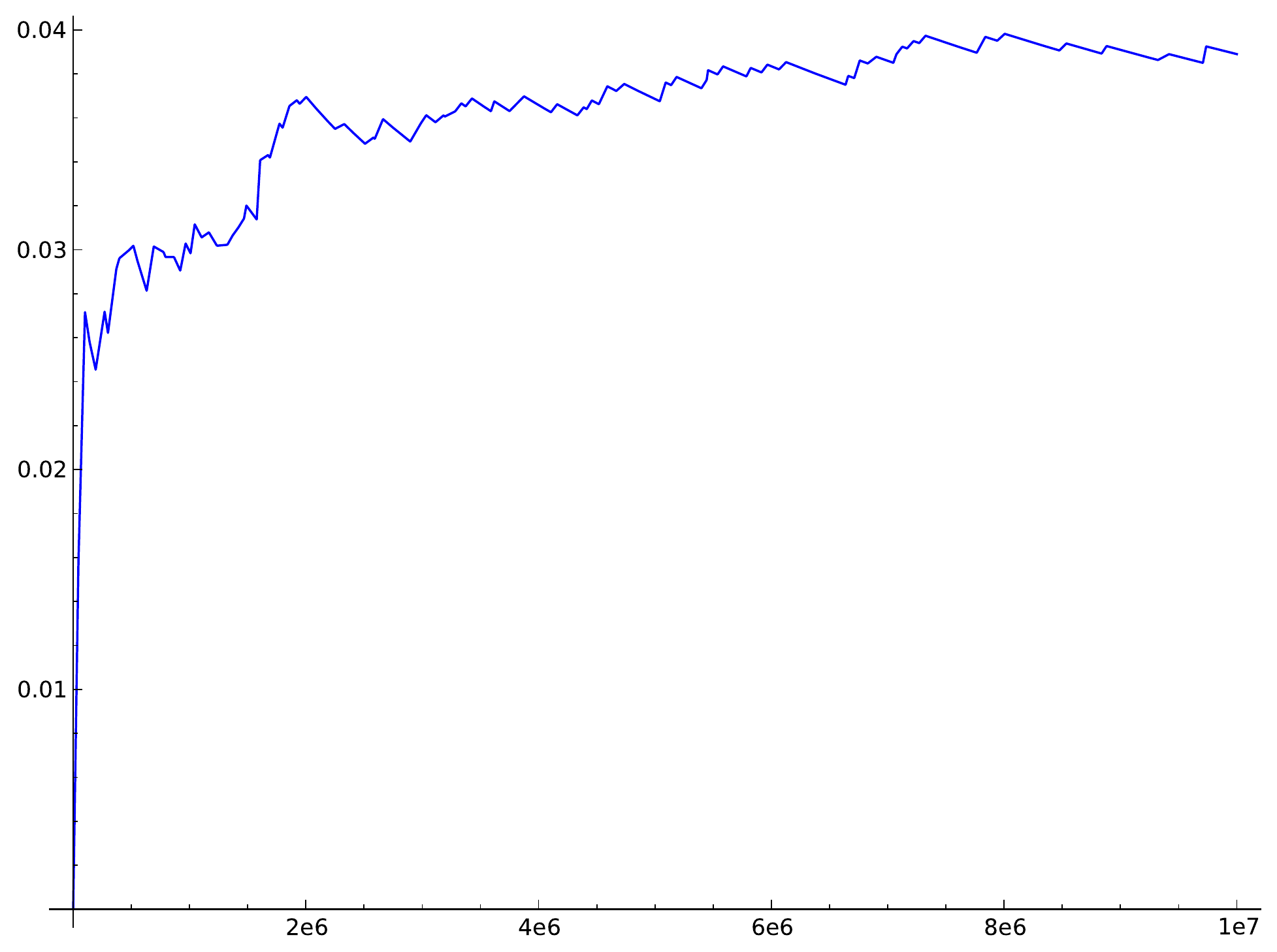}\hfill  \includegraphics[scale=.3]{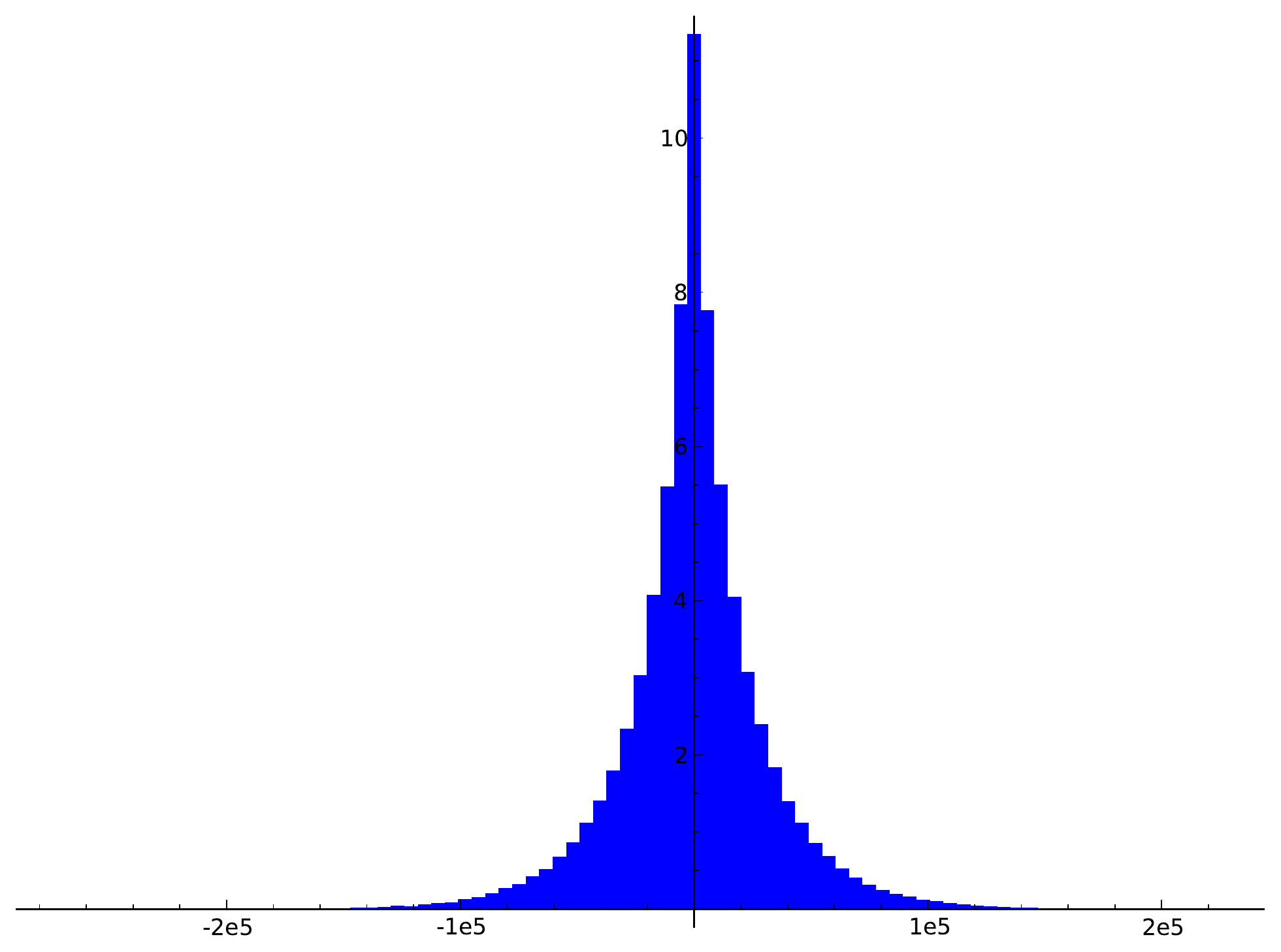}
\caption{$\#N_f(\Z_F,X) / (X^{1/4}(\log X)^{11/8})$ for the form \texttt{2.2.5.1-[4,4]-11.1-a} on the left and the distribution of the values of the coefficients $c_{|D|}(g)$ that corresponds to $f$ as in Conjecture~\ref{conj:Waldspurger}.}\label{fig:wt4} 
  \end{figure}

\section{Remaining questions} \label{sec:ques}

This paper just scratches the surface of this topic, and many questions remain.

\begin{enumerate}
\item What can be said about twists over a totally real field with other signatures?  In particular, are the asymptotics the same for discriminants $D$ such that $F(\sqrt{D})$ is of mixed signature?  (Such extensions are a genuinely new phenomenon over $F \neq \Q$.)

\item Do the same asymptotics apply when the narrow class number of $F$ is bigger than $1$, or does the class group pose an obstruction?

\item What happens when the modular form $f$ has coefficients in a field $K$ larger than $\Q$?  For weight $k=2$, such a form corresponds to an isogeny class of abelian varieties of $\GL_2$-type defined over $F$.  More care must be taken in the discretization step in this situation, since it must be performed with respect to the ring of integers of $K$ embedded as a lattice.

\item Given a half-integer weight (Hilbert) modular form $g$ (corresponding to a form of even weight $f$ under the Shimura correspondence), what is the distribution of $c_{|D|}(g)$, appropriately normalized?  
\end{enumerate}

\begin{acknowledgements}\label{ackref}
The authors would like to thank Brian Conrey and Ariel Pacetti for
helpful discussions, as well as Mark Watkins and the anonymous
referees for several comments and corrections.
\end{acknowledgements}
\newpage

\bibliography{rmt}
\bibliographystyle{alpha}


\affiliationone{
   Nathan C.\ Ryan\\
   Department of Mathematics \\ Bucknell University \\ Lewisburg, PA
17837 \\ 
USA 
\email{nathan.c.ryan@gmail.com}}
\affiliationtwo{
   Gonzalo Tornar\'ia \\
   Centro de Matem\'atica \\
   Universidad de la Rep\'ublica \\ 11400 Montevideo\\Uruguay
   \email{tornaria@cmat.edu.uy}}
\affiliationthree{%
   John Voight \\
   Department of Mathematics \\
   Dartmouth College \\
   6188 Kemeny Hall \\
   Hanover, NH  03755 \\
   USA
   \email{jvoight@gmail.com}}

\end{document}